\documentclass{birkjour}

\usepackage{textcmds} 
\usepackage{amsmath, amssymb, amsfonts, amstext, amsthm, amscd, mathrsfs, mathscinet}
\usepackage{mathtools}
\usepackage{verbatim}
\usepackage{graphicx}
\usepackage{color}
\usepackage{pinlabel}
\usepackage{mathrsfs} 
\usepackage{enumitem}
\usepackage{tikz}
\usepackage{tikz-cd}
\usetikzlibrary{matrix}

\usepackage{etoolbox}
\makeatletter
\let\ams@starttoc\@starttoc
\makeatother
\usepackage[parfill]{parskip}
\makeatletter
\let\@starttoc\ams@starttoc
\patchcmd{\@starttoc}{\makeatletter}{\makeatletter\parskip\z@}{}{}
\makeatother

\usepackage{hyperref}

\newcommand{\C}{\mathbb{C}}
\newcommand{\R}{\mathbb{R}}
\newcommand{\Z}{\mathbb{Z}}

\newcommand{\CP}{\mathbb{C}P}
\newcommand{\RP}{\mathbb{R}P}
\newcommand{\T}{\mathbb{T}}

\newcommand{\im}{\mathrm{im}}

\newcommand{\OP}{\operatorname}
\newcommand{\rank}{\operatorname{rank}}

\newcommand{\pt}{\operatorname{pt}}

\newsavebox{\textvisiblespacebox}
\savebox{\textvisiblespacebox}{\texttt{aa}}
\newcommand\vartextvisiblespace[1][\wd\textvisiblespacebox]{%
 \makebox[#1]{\kern.1em\rule{.4pt}{.3ex}%
 \hrulefill%
 \rule{.4pt}{.3ex}\kern.1em}%
}

\numberwithin{equation}{section}

\newtheorem{thm}{Theorem}[section]
\newtheorem{lma}[thm]{Lemma}
\newtheorem{prp}[thm]{Proposition}
\newtheorem{cor}[thm]{Corollary}
\newtheorem{mainthm}{Theorem}

\newtheoremstyle{TheoremNum}
 {\topsep}{\topsep}  
 {\itshape}   
 {}    
 {\bfseries}   
 {.}    
 { }    
 {\thmname{#1}\thmnote{ \bfseries #3}}
\theoremstyle{TheoremNum}

\theoremstyle{definition}
\newtheorem{dfn}[thm]{Definition}

\newtheorem{ex}[thm]{Example}

\theoremstyle{remark}
\newtheorem{rmk}[thm]{Remark}

\begingroup 
\makeatletter 
\@for\theoremstyle:=definition,remark,plain,TheoremNum\do{%
\expandafter\g@addto@macro\csname th@\theoremstyle\endcsname{%
\addtolength\thm@preskip\parskip 
}%
} 
\endgroup 

\title{On weakly exact Lagrangians in Liouville bi-fillings}

\author{Georgios Dimitroglou Rizell}
\address{Department of Mathematics\\
Uppsala University\\
Box 480\\
SE-751 06 UPPSALA\\
SWEDEN}
\email{georgios.dimitroglou@math.uu.se}
\thanks{The author is supported by the Knut and Alice Wallenberg Foundation through the grants KAW 2021.0191 and KAW 2023.0294.}

\begin{document}

\begin{abstract}
Here we study several questions concerning Liouville domains that are diffeomorphic to cylinders, so called trivial bi-fillings, for which the Liouville skeleton moreover is smooth and of codimension one; we also propose the notion of a Liouville-Hamiltonian structure, which encodes the symplectic structure of a hypersurface tangent to the Liouville flow, e.g.~ the skeleta of certain bi-fillings. We show that the symplectic homology of a bi-filling is non-trivial, and that a connected Lagrangian inside a bi-filling whose boundary lives in different components of the contact boundary automatically has non-vanishing wrapped Floer cohomology. We also prove geometric vanishing and non-vanishing criteria for the wrapped Floer cohomology of an exact Lagrangian with disconnected cylindrical ends. Finally, we give homotopy-theoretic restrictions on the closed weakly exact Lagrangians in the McDuff and torus bundle Liouville domains.
\end{abstract}

\maketitle
\setcounter{tocdepth}{1}
\tableofcontents

\section{Introduction and results}

A compact \textbf{Liouville domain} is a pair $(X^{2n},\lambda)$ consisting of a smooth $2n$-dimensional manifold $X$ with smooth boundary $\partial X \ne \emptyset$, a one-form $\lambda \in \Omega^1(X)$ for which $d\lambda$ is symplectic, and for which the Liouville vector field $\zeta$ defined by $\iota_\zeta \omega=\lambda$ is everywhere outwards pointing along the boundary. One says that the Liouville domain $(X^{2n},\lambda)$ is a \textbf{Liouville filling} of its contact boundary $(\partial X,\lambda|_{T\partial X})$.

The most well-studied case of Liouville domains are those that are Weinstein domains, i.e.~to a Liouville domains for which $\zeta$ is gradient-like for a Morse function, or exact deformations $(X^{2n},\lambda+df)$ of such domains; we refer to \cite{Cieliebak:SteinWeinstein} by Cieliebak--Eliashberg for the precise definition. In the case when $(X,\lambda)$ arises from a Weinstein structure, the skeleton
$$ \OP{Skel}(X^{2n},\lambda) \coloneqq \bigcap_{t \ge 0} \phi^{-t}_{ \zeta}(X)$$
can be seen to admit an isotropic structure. In particular, this means that a $2n$-dimensional Weinstein domain has a handle-decomposition consisting of cells of index at most $n$. In particular, from this we conclude that the natural map $H_0(\partial X) \to H_0(X)$ induced by the inclusion of the boundary is an isomorphism whenever $n \ge 2$ and $X$ is Weinstein.

Comparatively little is known about Liouville domains that admit no deformations to a Weinstein domain through Liouville structures. All known examples are basically of the form $X^{2n}=M^{2n-1} \times I$ (possibly with additional Weinstein handles added) for $n \ge 2$. The first examples were constructed by McDuff in \cite{McDuff} and it was later discovered that they fit into a more general framework of Liouville structures arising from Anosov dynamics. Since such a Liouville domain fills two contact manifold, but in a topologically trivial way, they were called \textbf{trivial Liouville bi-fillings} in \cite{Hozoori}. The link between the known trivial bi-fillings examples and Anosov flows on three-dimensional manifolds was first exhibited in work by Mitsumatsu \cite{Mitsumatsu}. We refer to more recent work by Hozoori \cite{Hozoori} and Massoni \cite{Massoni:Anosov} for the latest developments of the connections between Anosov dynamics, bi-contact structures, and Liouville bi-fillings. 

In this article we only consider trivial Liouville bi-fillings whose skeleton is smooth and of codimension one. We discuss some general aspects of such Liouville manifolds in Section \ref{sec:Liouville}; there we introduce a concept called a \text{Liouville-Hamiltonian structure} in order to capture the behaviour of a Liouville vector-field that is tangent to a codimension one hypersurface, e.g.~the skeleton in the cases of interest. It should be noted that we do not know any examples of trivial Liouville bi-fillings with a smooth codimension one skeleton beyond the cases that correspond to Anosov flows as studied by the aforementioned authors. Here our main focus will be the most well studied cases; those of the McDuff domains (Subsection \ref{sec:McDuff}) and torus bundle domains (Subsection \ref{sec:torus}).

Anosov flows are very rich dynamical systems, and the symplectic topology of the corresponding trivial Liouville bi-fillings encodes this. In \cite{Cieliebak:Anasov} Cieliebak--Lazarev--Massoni--Moreno showed that the wrapped Fukaya category in the case of a McDuff or torus bundle domain contains information about the dynamics of the related Anosov flows. Since our results and proofs do not rely on this category, we will not give a full description of it, but instead refer to the aforementioned work for details. Roughly, the wrapped Fukaya category is a category whose objects are the exact Lagrangians $L^n \subset (X^{2n},\lambda)$ with \textbf{cylindrical ends}. These are half-dimensional smooth and compact submanifolds possibly with boundary on which $\lambda$ restricts to an exact one-form, whose boundary $\partial L \subset \partial X$ is contained in the boundary of $\partial X$, and where a collar of $\partial L$ is tangent to the Liouville flow. This means that the boundary $\partial L \subset (\partial X,\lambda|_{T\partial X})$ is a Legendrian submanifold of the contact boundary, and that $\lambda|_{TL} \in \Omega^1(L)$ vanishes near $\partial L$ and is globally exact on $L$. Under the weaker assumptions that $\lambda|_{TL}$ merely is a closed one-form that vanishes near the boundary, and for which the symplectic area of any class in $\pi_2(X,L)$ vanishes, the Lagrangian submanifold $L$ is said to be \textbf{weakly exact}.

The wrapped Fukaya category is one of the most powerful invariants of the symplectic topology of a Liouville domain. However, in order to have any chance to compute the invariant, we need a well-behaved set of generators of this category. The reason is that the set of cylindrical Lagrangians, and thus objects of this category, has infinite cardinality, even when they are considered up to Hamiltonian isotopy. Furthermore, the set of cylindrical Lagrangians up to Hamiltonian isotopy is not possible to grasp. In the case of a Weinstein domain the wrapped Fukaya category has been shown to be generated by the finite number of Lagrangian cocore discs for any choice of Weinstein handle decomposition; see work by the author with Chantraine--Ghiggini--Golovko \cite{Generation} as well as work by Ganatra--Pardon--Shende \cite{GPS}. In the case of a non-Weinstein Liouville domain, no analogous result. It should be noted that even in the case of McDuff or torus bundle domains, the existence of a finite generating set seems implausible in the light of \cite{Cieliebak:Anasov}. 

The main goal of this paper is to provide new restrictions on the behaviour of weakly exact Lagrangian submanifolds in Liouville bi-fillings, mainly in the aforementioned cases of a McDuff or torus-bundle domain.

Our first result concerns a topological restriction on the closed weakly exact Lagrangians inside the McDuff and torus-bundle domains. We refer to Subsections \ref{sec:McDuff} and \ref{sec:torus} for the precise definitions of these trivial Liouville bi-fillings. Here it is just important to recall the following particularities of these domains.
\begin{itemize}
 \item \emph{Torus-bundle domains} are the total spaces of a Lagrangian $\T^2$-fibrations $p \colon V \to S^1 \times I$, whose torus fibres all are weakly exact. These will be called \textbf{standard weakly exact Lagrangian tori}.
 \item \emph{McDuff domains} are total spaces of non-trivial $S^1$-bundles $p \colon V \to \Sigma_g \times I$ for $g \ge 2$ of Euler number $2g-2$. The standard weakly exact Lagrangian tori of McDuff domains are the \textbf{circle-bundle tori} described in \cite[Remark 4.3]{Cieliebak:Anasov}; these tori are homotopic to a union of $S^1$-fibres of the bundle over a homotopically non-trivial curve in the base. It was shown in \cite[Theorem 3]{Cieliebak:Anasov} that when the aforementioned curve in $\Sigma_g$ is a closed embedded geodesic, then the circle-bundle torus can be deformed through weakly exact embedded Lagrangian tori to one that is exact. We will call the latter an \textbf{exact circle-bundle torus}.
\end{itemize} 
 In particular, the torus-bundle domains and McDuff domains are smooth fibre-bundles over the base $S^1$ and $\Sigma_g$, respectively. Frequently we will take both finite and infinite covers of these bases to induce covers $\tilde{V} \to V$ of the corresponding Liouville domain $V$. 

Note that all tori described above are incompressible. Our first result is that all weakly exact, respectively exact, Lagrangian tori in the two above domains are homotopic to standard tori.
\begin{mainthm}
\label{thm:a}
Let $L^2 \subset (V^4,d\lambda)$ be a closed Lagrangian submanifold of either a McDuff or a torus-bundle domain. Then $L$ is either a torus or an connected sum of $2k$ number of $\RP^2$'s with $k \ge 2$. When $L$ is a weakly exact torus, it moreover follows that
\begin{enumerate}
\item \emph{When $V \to S^1 \times I$ is a torus bundle domain:} The Lagrangian $\iota \colon L \hookrightarrow V$ lifts to any cover $\tilde{V}_k\to V$ that is induced by a $k$-fold cover of the base $S^1$. Moreover, when $k \gg 0$ is sufficiently large, the lifted Lagrangian $\tilde\iota \colon L \hookrightarrow \tilde{V}_k$ is Hamiltonian isotopic to a standard fibre in the Lagrangian torus-bundle $\tilde{V}_k$.
\item \emph{When $V \to \Sigma_g$ is a McDuff domain:} There exists a finite $k$-fold cover $\tilde{V}_k \to V$ that is induced by some suitable $k$-fold cover of the base $\Sigma_g$, which hence again is a McDuff domain, under which $\iota \colon L \hookrightarrow V$ lifts to a Lagrangian $\tilde\iota \colon L \hookrightarrow \tilde{V}_k$ that is isotopic to a circle-bundle torus through weakly exact Lagrangians. When $L$ is exact, then we can moreover assume that there exists a Hamiltonian isotopy to an exact circle-bundle torus inside $\tilde{V}_k$.
\end{enumerate}
In either case, the inclusion $L \subset V$ of a weakly exact Lagrangian torus is incompressible.
\end{mainthm}

We currently do not have any results on the generation of the wrapped Fukaya category of Liouville bi-fillings. In the process of showing that some given subset of Lagrangian submanifolds generate the wrapped Fukaya category, it is useful to have geometric conditions for which objects are vanishing or non-vanishing in the category. E.g.~the fact that any exact Lagrangian that is disjoint from the skeleton has vanishing wrapped Floer cohomology was a crucial ingredient in the proof of the generation criterion in \cite{Generation}. In Liouville bi-fillings we instead have a natural condition for \emph{non-vanishing} of the wrapped Floer cohomology group.

Assume that $(X,\lambda)$ is a connected Liouville domain whose contact boundary consists of the connected components 
$$ \partial X =\bigsqcup_{i \in \pi_0(\partial X)} (\partial X)_i.$$
For any Lagrangian $L \subset X$ with cylindrical ends, we consider the decomposition
$$(\partial L)_i \coloneqq \partial L \cap (\partial X )_i, \:\: i \in \pi_0(\partial X)$$
induced by the components of the contact boundary. Denote by $I_L \subset \pi_0(\partial X)$ the image of the natural map $\pi_0(\partial L) \to \pi_0(\partial X)$ induced by the inclusion.

 In the following two results we use the grading convention of \cite{Ritter:TQFT}, in which wrapped Floer cohomology $HW^*(L,L)$ and symplectic cohomology $SH^*(X)$ are graded (in general over $\Z_2$) unital algebras that admit graded canonical maps from the Morse (or singular) cohomology rings $H^*(L)$ and $H^*(X)$, respectively. 
\begin{mainthm}
\label{thm:b}
Let $L^n \subset (X^{2n},\lambda)$ be a connected exact Lagrangian with cylindrical ends. The kernel of the canonical map in wrapped Floer cohomology
$$ H^*(L) \to HW^*(L,L) $$
is contained in the image of the canonical map in singular cohomology
$$ \bigoplus_{i \in I_L} H^*(L,\partial L\setminus (\partial L)_i) \to H^*(L).$$
In particular, when $|I_L| \ge 2$ (i.e.~when $L$ has boundary components contained in several different components of $\partial X$), the latter map does not hit $H^0(L)$, and thus the wrapped Floer cohomology $HW^*(L,L) \neq 0$ is non-vanishing.
\end{mainthm}
There is a completely analogous result also for the versions of Floer homology defined for periodic Hamiltonian orbits on a Liouville domain.
\begin{mainthm}
\label{thm:b2}
Let $(X,\lambda)$ be a connected Liouville domain. The kernel of the canonical map in symplectic cohomology
$$ H^*(X) \to SH^*(X) $$
is contained in the image of the canonical map in singular cohomology
$$ \bigoplus_{i \in \pi_0(\partial X)} H^*(X,\partial X \setminus (\partial X)_i) \to H^*(X).$$
In particular, when $|\pi_0(\partial X)| \ge 2$, the latter map does not hit $H^0(X)$, and thus symplectic cohomology $SH^*(X) \neq 0$ is non-vanishing.
\end{mainthm}
The latter two theorems give an immediate topological condition for when the Floer homology is non-trivial. In Theorem \ref{thm:vanishing} and Corollary \ref{cor:vanishing} we also give vanishing conditions for wrapped Floer homology. The underlying mechanism is based upon the vanishing of wrapped Floer homology in the presence of positive contractible loops of its Legendrian boundary; c.f.~work by Chantraine--Colin and the author \cite{Chantraine:Positive} as well as Cant--Hedicke--Kilgore \cite{HedickeCantKilgore}. 

The non-vanishing result for the wrapped Floer cohomology given by Theorem \ref{thm:b} implies constraints on the dynamics of the Liouville flow on the skeleton in the case when the skeleton is smooth, codimension one, and nowhere characteristic; see Definition \ref{dfn:nowherechar}. Namely, given periodic orbits of the Liouville flow Corollary \ref{cor:Lagcyl} produces an exact Lagrangian cylinder with cylindrical ends in two different contact boundary components. Theorem \ref{thm:b} then implies that the wrapped Floer homology of such a Lagrangian is non-vanishing. On the other hand, if the orbit is contained in a smooth ball, we can use the vanishing result Corollary \ref{cor:vanishing} to show that the wrapped Floer homology vanishes. In other words:
\begin{mainthm}
\label{thm:c}
 Let $M^3 \subset (X^4,\lambda)$ be a closed hypersurface of a Liouville domain tangent to the Liouville vector field $\zeta_\lambda$, where the latter is assumed to be:
\begin{itemize}
\item nowhere characteristic along $M$ (see Definition \ref{dfn:nowherechar}); and
\item $[\nu]^*$-repelling along $M$ for some choice of non-vanishing normal vector $\nu\in\ker\lambda$ to $M$ (see Definition \ref{dfn:skeleton}).
\end{itemize}
Then no smooth ball contains a periodic orbit of the Liouville vector field.
\end{mainthm}
\begin{rmk}
In higher dimensions there is no immediate obstruction to the existence of periodic orbits of the Liouville flow on the skeleton. However, analogously to the above result, there are obstructions to the existence of a closed $n-1$-dimensional submanifold of a $2n-1$-dimensional hypersurface $M \subset (X^{2n},\lambda)$ tangent to the Liouville vector field (which is assumed to satisfy the properties of the theorem), where the submanifold
\begin{itemize}
 \item is transverse to the characteristic foliation;
 \item pulls back the Liouville form to zero;
 \item lives in a smooth ball that becomes a Darboux ball in nearby contact-type hypersurfaces.
 \end{itemize}
\end{rmk}

 It is a non-trivial task to construct examples where the assumptions of Theorem \ref{thm:c} are satisfied. All examples known to the author come from work of Mitsumatsu \cite{Mitsumatsu}, then further developed by Hozoori \cite{Hozoori}, \cite{Hozoori:Skeleton} and Massoni \cite{Massoni:Anosov}. In certain particularly well-behaved cases their constructions yield Liouville vector fields that are Anosov. In that case something even stronger is true: the Liouville flow has no periodic orbits that are contractible. 

\section*{Acknowledgements}
 I am grateful to Surena Hozoori for pointing out a mistake in the first version of this paper and for the very thorough comments of the referee (in the first version additional assumptions were missing to ensure that the Liouville vector field tangent to a hypersurface is repelling). Then, I would like to the organisers of the wonderful conference \emph{Symplectic Geometry and Anosov Flows} in Heidelberg 22 -- 26 July 2024; Peter Albers, Jonathan Bowden and Agustin Moreno. I would also like to thank Thomas Massoni for useful discussions. Interactions at this conference gave rise to lots of inspiration and ideas that were crucial for this paper.

\section{Trivial Liouville bi-fillings and Liouville-Hamiltonian structures}

 Here we provide a general framework for studying and classifying Liouville flows near a compact hypersurface $M^{2n-1} \subset (X^{2n},\lambda)$ that is invariant under the Liouville flow (and which hence is a subset of the skeleton). The goal is to characterise trivial Liouville bi-fillings $(X=M \times [-1,1],\lambda)$ whose skeleton $M \subset X$ is smooth and of codimension one. In order to do this we need to analyse the Liouville structure near a hypersurface that is tangent to the Liouville vector field. Theorem \ref{thm:classification} provides a type of standard-form for these hypersurfaces under certain additional assumptions. This naturally leads to the abstract notion of a so-called Liouville-Hamiltonian structure, that we define in Subsection \ref{sec:Liouville}, which is a structure that encodes the Liouville structure near a hypersurface that is tangent to the Liouville flow. 

A smooth skeleton of a Liouville domain is homotopy equivalent to one of its tubular neighbourhoods, and hence homotopy equivalent to the domain itself. It follows that the skeleton is orientable, and that it has a smooth trivial tubular neighbourhood $M \times [-1,1] \hookrightarrow X$. In favourable situations the entire Liouville manifold is a trivial product $X\cong M \times [-1,1]$ with the skeleton included as $M \times \{0\} \subset M \times [-1,1] = X$, and with the Liouville vector field $\zeta_\lambda$ having a non-vanishing $\partial_s$-component outside of the skeleton. In order to formulate conditions for when this is the case, we need to introduce some additional properties; in particular see Definitions \ref{dfn:nowherechar} and \ref{dfn:skeleton}. 

Constructions of four-dimensional Liouville manifolds $X^4=M^3 \times [-1,1]$ with skeleton of the above type go back to Mitsumatsu's work \cite{Mitsumatsu}. More recent developments were carried out by Hozoori in \cite{Hozoori:Skeleton} and \cite{Hozoori}, and Massoni \cite{Massoni:Anosov}. The common feature of these constructions is that $X^4=M^3 \times [-1,1]_s$ is endowed with the structure of a Liouville domain by interpolating between two contact forms $\alpha_\pm$ on $M$. Here we will encounter Liouville forms $\lambda=(1-s)\alpha_-+(1+s)\alpha_+$, which are Liouville under conditions on the contact forms $\alpha_\pm$. In \cite[Theorem 1.2]{Hozoori:Skeleton} Hozoori gives a dynamical characterisation those flows on $M$ that can be extended to a Liouville flow of this type.

\subsection{ A standard form for hypersurfaces tangent to the Liouville flow}

We begin with the following definition.

\begin{dfn}
\label{dfn:nowherechar}
Let $M \subset (X,\lambda)$ be a smooth hypersurface. The Liouville vector field $\zeta_\lambda \in \Gamma(TX)$ is said to be \emph{nowhere characteristic on $M$} if it is tangent to $M$, and there exists a smooth and non-vanishing normal vector field $\nu \in \ker (\lambda|_M) \subset TX|_M$ to $M$.
\end{dfn}
 
If $\zeta_\lambda$ is non-vanishing on $M$, then being nowhere characteristic on $M$ is equivalent to the intersection $\ker \lambda \pitchfork TM$ being transverse at every point in $M$. We recall Hozoori's result that provides several different characterisations of this condition: 
\begin{lma}[Lemma 7.1 in \cite{Hozoori:Skeleton}]
 Assume that the Liouville vector field $\zeta_\lambda$ on $(X,\lambda)$ is tangent to a hypersurface $M$, and that it moreover is non-zero there. Then the following statements are equivalent:
 \begin{enumerate}
 \item $\zeta_\lambda$ is nowhere characteristic on $M$ in the sense of Definition \ref{dfn:nowherechar};
 \item $\zeta_\lambda$ is not tangent to the characteristic distribution $\ker d\lambda|_M$ of $M$ at any point; and
 \item the intersection $\ker(\lambda|_M) \pitchfork TM$ of codimension-one sub-bundles of $TX$ is transverse at every point in $M$.
 \end{enumerate}
\end{lma}
 
An important consequence of $\zeta_{ \lambda}$ being nowhere characteristic on $M$ is that the pull-back $\eta \coloneqq \lambda|_{TM} \in \Omega^1(M)$ has the same rank as $\lambda|_M$. In other words, $\lambda$ and $\eta$ agree along $M$ after extending the latter form $\eta$ by zero in the direction $\nu$. The following proposition elaborates on this property: 
\begin{prp}
\label{prp:beta}
 
Assume that the Liouville vector field $\zeta_\lambda \in \Gamma(TX)$ of $(X,\lambda)$ is nowhere characteristic on the hypersurface $M \subset X$, and let $\nu \in \ker(\lambda|_M)$ be a choice of non-vanishing normal vector field along $M$.
\begin{enumerate}
 \item For any choice of embedding
$$ \iota \colon M \times \R_s \hookrightarrow X$$
such that $M=\{s=0\}$ and $\partial_s|_M=\nu$, we have
$$\iota^*(d\lambda)|_{M \times \{0\}}=d\eta+ds \wedge \beta_\nu,$$
where the one-forms $\eta=\lambda|_{TM},\beta_\nu \in \Omega^1(M)$ are extended to all of $M \times \R_s$ by pull-back under the canonical projection $M \times \R_s \to M$. Note that $\beta_\nu(\zeta_\lambda)=0$ while $\beta_\nu|_{\ker d\eta} \neq 0$ is nowhere vanishing.

\item For $\nu'=\pm e^g\nu+\Xi$ where $\Xi \in \ker (\lambda|_{TM}) \subset TM$ and $g \colon M \to \R$ is smooth, we have
$$ \beta_{\nu'}=\pm e^g\beta_\nu+\iota_\Xi d\eta,$$
where the second term thus vanishes on both $\zeta_\lambda$ and $\ker d\eta$.
\end{enumerate}
\end{prp}
\begin{rmk}
 
We will see that the quantity $d\beta(\mathcal{C},\zeta_\lambda)$ determines the infinitesimal properties of the dynamics of the Liouville flow near $M$, where $\mathcal{C} \in \Gamma(TM)$ is a generator of $\ker d\eta \subset \Gamma(TM)$ that satisfies $\beta_\nu(\mathcal{C})=1$. Note that Part (2) of Proposition \ref{prp:beta} implies the identity
$$d\beta_{\nu'}(\pm e^{-g}\mathcal{C},\zeta_\lambda)=dg(\zeta_\lambda)+d\beta_\nu(\mathcal{C},\zeta_\lambda)$$
between the two-forms induced by the two different choices of normals $\nu$ and $\nu'=\pm e^g\nu+\Xi$ along $M$, where $\Xi \in \ker (\lambda|_{TM})$.
\end{rmk}
\begin{proof}
 
(1.) We can clearly write
$$\iota^*(d\lambda|_M)=\omega+ds \wedge \beta_\nu$$
where both $\omega \in \Omega^2(M)$ and $\beta_\nu \in \Omega^1(M)$ have been extended to all of $M \times \R_s$ by the pull-back of the canonical projection $M \times \R_s \to M$. Clearly, $\omega$ is independent of any choices made, while $\beta_\nu$ only depends on the normal vector field $\nu$. Furthermore, we clearly have $\omega|_{TM}=d\eta$ where $\eta=\lambda|_{TM}$, while $\partial_s \in \ker (\lambda|_M)$ implies that $\beta_\nu(\zeta_\lambda)=0$. Hence, the pull-back of the symplectic form is as claimed.

(2.) For a different choice $\iota' \colon M \times \R_{s'} \hookrightarrow X$ with coordinate $s'$ and $\partial_{s'}|_M=\nu' \in \ker \lambda$ with $\nu'=\pm e^g\nu+\Xi$, $g \colon M \to \R$ smooth and $\Xi \in \ker (\lambda|_{TM}) \subset \Gamma(TM)$, we get
$$(\iota')^*d\lambda|_M=d\eta+ds' \wedge \beta_{\nu'}.$$
We can assume that $\iota^{-1}\circ \iota'$ is well-defined near $M$, and thus get
$$d\eta+ds' \wedge \beta_{\nu'}=(\iota^{-1}\circ \iota')^*(d\eta+ds\wedge\beta_\nu)=d\eta+ds' \wedge (\pm e^g\beta_\nu+\iota_{\Xi}d\eta)$$
as sought, since
\begin{gather*}
(\iota^{-1}\circ \iota')^*d\eta=d\eta+ds' \wedge \iota_\Xi d\eta,\\
(\iota^{-1}\circ \iota')^* ds \wedge \beta_\nu=\pm e^g ds' \wedge \beta_\nu,
\end{gather*}
are satisfied.
\end{proof}
 
As we will see, the one-form $\beta_\nu$ contains information about the infinitesimal growth of the $\nu$-component of $\zeta_\lambda$ measured along the same normal direction; see Theorem \ref{thm:classification} for a precise statement.

In the following we let $N^*M \subset T^*X|_M \to M$ denote the co-normal bundle of the orientable hypersurface $M \subset X$, i.e.~the one-dimensional sub-bundle of one-forms on $TX|_M$ that vanish on $TM$. Since any non-vanishing vector field $\nu$ that is normal to $M$ induces a fibre-wise basis $[\nu] \in \Gamma(TX|_M/TM)$ of the normal bundle, we get an induced fibre-wise basis $[\nu]^* \in \Gamma(N^*M)$ of the conormal bundle by the requirement $[\nu]^*([\nu])=1$ at every fibre.
 
\begin{dfn}
\label{dfn:skeleton}
Assume that the Liouville vector field $\zeta_\lambda$ on $(X,\lambda)$ is tangent to a hypersurface $M$, and let $\alpha$ be a non-vanishing section of its co-normal bundle $N^*M \subset T^*X|_M$. We say that $\zeta_\lambda$ is \emph{$\alpha$-repelling along $M$} if $$d(ds(\zeta_\lambda))|_M=e^g\alpha,$$
holds for some $g \colon M \to \R$ for any $s \colon X \to \R$ that satisfies $ds|_M=\alpha.$
\end{dfn}
 
Two functions $\tilde{s}$ and $s$ whose differentials $d\tilde{s}|_M=ds|_M$ agree along $M$ satisfy $$d\tilde{s}=ds+d(s^2G)=ds+2sG\,ds+s^2\,dG, \:\: G \in C^\infty(X).$$
It follows that $d\tilde{s}(\zeta_\lambda)=ds(\zeta_\lambda)+\mathbf{O}(s^2)$ which means that the one-form $d(ds(\zeta_\lambda))|_M$ along $M$ only depends on the choice of $\alpha=ds|_M$.

Any choice of vector field $\nu \in \ker \lambda$ that is normal to $M$ can be extended to an embedding $\iota \colon M \times \R_s \to X$ with $\partial_s=\nu$ and $\iota^{-1}(M)=M \times \{0\}$. There is an identity $[ds]=[\nu]^*$ of sections of the co-normal bundle $N^*M$ of $M \subset X$.

We have the following normal form for a Liouville vector field that is tangent to $M$ and nowhere characteristic in the sense of Definition \ref{dfn:nowherechar}. In particular, it follows from this result that the $[\nu]^*$-repelling property of $\zeta_\lambda$ can be expressed in terms of the two-form $d\beta_\nu$.
\begin{thm}
\label{thm:classification}
Assume that the Liouville vector field $\zeta_\lambda \in \Gamma(TX)$ of the Liouville domain $(X^{2n},\lambda)$ is nowhere characteristic on the hypersurface $M^{2n-1}$. Let $\nu \in \ker \lambda|_M$ be a choice of non-vanishing vector field normal to $M$. There is an embedding $\iota \colon M \times [-\epsilon,\epsilon]_s \hookrightarrow X$, for $\epsilon>0$ sufficiently small, $M=s^{-1}(0)$, and $\partial_s|_M=\nu$, such that
$$\iota^*\lambda=\eta+s\beta_\nu+d(s^2G)$$
for $\eta,\beta_\nu \in \Omega^1(M)$ as in Proposition \ref{prp:beta}, and $G \in C^{\infty}(M \times [-\epsilon,\epsilon])$. Moreover;
\begin{enumerate}
 \item In these coordinates the Liouville vector field can be written as
$$\zeta_\lambda=\zeta_\eta+2sG\cdot \mathcal{C}+s\cdot \Xi+sF\cdot \partial_s,$$
where $\zeta_\eta \in \Gamma(TM)$ satisfies $\iota_{\zeta_\eta}d\eta=\eta$, $\Xi \in \cap \ker \beta_\nu \cap \Gamma(TM)$, and $\mathcal{C} \in \ker d\eta$ is defined uniquely by $\beta_\nu(\mathcal{C})=1$.
\item The function $F$ is determined by $$F=1+d\beta_\nu(\mathcal{C},\zeta_\eta).$$
In particular, $\zeta_\lambda$ is $[ds]$-repelling if and only if $d\beta_\nu(\mathcal{C},\zeta_\eta)>-1$ holds on $M$.
\end{enumerate}
\end{thm}
\begin{proof}
By Proposition \ref{prp:beta} we have the identity
$$\iota^*d\lambda|_M = d\eta + ds \wedge \beta_\nu,\:\: \eta, \beta_\nu \in \Omega^1(M). $$
We can use the standard symplectic neighbourhood theorem to replace the embedding with one that agrees with $\iota$ along $M \times \{0\}$ to the first order, and for which
$$\iota^*d\lambda|_M = d(\eta + s \wedge d\beta_\nu)=d\eta +ds \wedge \beta_\nu +s\,d\beta_\nu$$
is satisfied in a neighbourhood of $M \times \{0\}$. Indeed, the latter two-form is symplectic near $M \times \{0\}$ and agrees with $\iota^*d\lambda$ along the hypersurface.

In these coordinates we express the Liouville vector field as 
$$\zeta_\lambda=\zeta_\eta+s\left(2H\cdot \mathcal{C}+\Xi+F\cdot\partial_s\right)+\mathbf{O}(s^2),$$
with $\Xi$ as above, simply from the fact that $\zeta_\lambda=\zeta_\eta$ along $M$. Note that $\lambda$ coincides with the Liouville form $\eta+s\beta_\nu$ along $M$. By the Poincar\'{e} lemma we conclude that $\iota^*\lambda=\eta+s\beta_\nu+d(s^2G)$, where the last term is exact and vanishing along $M$. In other words, by contracting $d\eta+ds\wedge \beta_\nu+s\,d\beta_\nu$ with $\zeta_\lambda$ we get
$$ \eta+s\beta_\nu+d(s^2G)=
\iota_{\zeta_\lambda}d\lambda= \eta+s(2H\,ds+\iota_\Xi d\eta+F\beta_\nu+\iota_{\zeta_\eta}d\beta_\nu)+\mathbf{O}(s^2),$$
from which we conclude that $H=G$ and
$$ \beta_\nu = \iota_\Xi d\eta+F\beta_\nu+\iota_{\zeta_\eta}d\beta_\nu.$$
Contracting this equation with $\mathcal{C}$ we obtain
$1=F+d\beta_\nu(\zeta_\eta,\mathcal{C}),$ from which the second claim follows.
\end{proof}

\subsection{ Skeleta with $[ds]$-repelling Liouville vector fields}
Recall that the \textbf{completion} of a Liouville domain is obtained as follows. A Liouville domain $(X,\lambda)$ has a collar which is isomorphic to the half-symplectisation
$$((-\infty,0]_\tau \times Y,e^\tau\alpha), \:\: Y=\partial X,\: \alpha=\lambda|_{TY}.$$
The completion $(\hat{X},\lambda)$ is obtained by adjoining the other half
$$ (\hat{X},\lambda)=(X,\lambda) \: \cup \: ((0,+\infty)_\tau \times Y,e^\tau\alpha)$$
of the symplectisation. Also, recall the standard fact that a symplectomorphism 
$\Phi \colon (\hat{X}_0,\lambda_0) \to (\hat{X}_1,\lambda_1)$
between two completions of Liouville domains that preserves the Liouville forms outside of a compact subset if and only if it is \textbf{cylindrical} outside of a compact subset of the cylindrical ends
$$ ([0,+\infty)\times Y_i,e^\tau\alpha_i) \subset (\hat{X}_i,\lambda_i),$$
 by which we mean that the symplectomorphism $\Phi$ takes the form
$$ \Phi(\tau,y) = (\tau-g(y),\phi(y)),$$
for $\tau \gg 0$. Hence, the component $\phi \colon (Y_0,\alpha_0) \to (Y_1,\alpha_1)$
is a contactomorphism that satisfies $\phi^*\alpha_1=e^g\alpha_0$. A symplectomorphism that merely preserves the Liouville forms up to the addition of an exact one-form is called \textbf{exact}.

Now consider the setting of Theorem \ref{thm:classification} above, in which $\eta,\beta \in \Omega^1(M)$ are one-forms for which $\lambda=\eta+s\beta+d(s^2G)$ is a Liouville form on $M \times [-\epsilon,\epsilon]_s$ for some $G \in C^\infty(M \times [-\epsilon,\epsilon])$ when $\epsilon>0$ is sufficiently small. In case when $\zeta_\lambda$ is $[ds]$-repelling, we can moreover assume that the Liouville vector-field is outwards transverse to the boundary when $\epsilon>0$, i.e.~$M \times [-\epsilon,\epsilon]$ is a Liouville domain. In this case, for each $t\in[0,1]$, we moreover get a smooth family 
$$(X_\epsilon,\lambda_t)=(M \times [-\epsilon,\epsilon]_s,\lambda_t), \:\: \lambda_t \coloneqq\eta+s\beta+d(t\cdot s^2G)$$
of Liouville domains parametrised by $t \in [0,1]$, where the completion of $(X_\epsilon,\lambda_1)$ coincides with the completion $(\hat{X},\lambda)$. Using the standard result Lemma \ref{lma:Liouvillefam} formulated below we obtain the following corollary of Theorem \ref{thm:classification}: 
\begin{cor}
\label{cor:completions}
 Assume that the Liouville vector field $\zeta_\lambda \in \Gamma(TX)$ of the Liouville domain $(X^{2n},\lambda)$ is nowhere characteristic on its skeleton, which is the smooth hypersurface $M^{2n-1} \subset X^{2n}$. If $\nu \in \ker \lambda|_M$ is a non-vanishing normal to $M$ for which $\zeta_\lambda$ is $[\nu^*]$-repelling along $M$, then there is an exact symplectomorphism, which is cylindrical at infinity, from the completion $(\hat{X},\lambda)$ to the completion of the Liouville domain
$$(M \times [-\epsilon,\epsilon]_s,\eta+s\beta_{ \nu}),\:\: \eta,\beta_{ \nu} \in \Omega^1(M)$$
 whose contact boundary consists of the two components $ (M,\eta\pm \epsilon \beta_\nu).$ Here $\epsilon>0$ is sufficiently small, while $\eta,\beta_\nu \in \Omega^1(M)$ are as in Theorem \ref{thm:classification}. 
\end{cor}
The following standard result is needed in the above corollary, and is included for completeness.
\begin{lma}
\label{lma:Liouvillefam}
Let $(X,\lambda_t)$, $ t \in [0,1],$ be a smooth family of Liouville domains. It follows that the completions $(\hat{X},\lambda_t)$ are exact symplectomorphic by symplectomorphisms that are cylindrical outside of a compact subset, i.e. that preserve the Liouville forms there.
\end{lma}
\begin{proof}
Since the smooth family of pairs
$$ (Y=\partial X,\alpha_t), \:\: \alpha_t \coloneqq \lambda_t|_{TY},$$
defines a smooth family of contact one-forms on $Y$, Gray's stability theorem \cite{Geiges:Intro} can be used to produce a smooth isotopy $\psi_t$ that satisfies $\psi_t^*\alpha_t=e^{f_t}\alpha_0$ for some smooth family of functions $f_t \colon Y \to \R.$ We use the negative Liouville flow for $\lambda_t$ applied to $Y =\partial X \subset X$ to obtain a smooth family of Liouville-form preserving symplectic embeddings
$$ \Phi_t \colon ((-\infty,0] \times Y,e^\tau\alpha_t) \hookrightarrow (X,\lambda_t)$$
of collars of $\partial X$ for the different Liouville forms.

Consider the Liouville sub-domain $X_0 \subset (X,\lambda_0)$ that is bounded by the contact-type hypersurface $\Phi_0(\{A\} \times Y) \subset X$ for some $A \ll 0$ sufficiently small. If we define
\begin{gather*} \Psi_t \colon (\R_\tau \times Y,e^\tau\alpha_0) \to (\R_\tau \times Y,e^\tau\alpha_t),\\
(\tau,y)\mapsto(\tau-f_t(y),\psi_t(y)),
\end{gather*}
we obtain a Liouville form-preserving symplectic isotopy
$$\Phi_t \circ \Psi_t\circ \Phi_0^{-1} \colon (\mathcal{O}(\partial X_0),\lambda_0) \hookrightarrow (X,\lambda_t)$$
defined on an open neighbourhood $\mathcal{O}(\partial X_0)$ of $\partial(X_0) \subset X$. We can extend this map to a smooth isotopy $\tilde{\Phi}_t \colon X_0 \hookrightarrow X$ which pulls back $\lambda_t$ to $\lambda_0$ near the boundary $\partial X_0$.

There is symplectic isotopy of embeddings $\tilde{\Phi}_t \circ \psi_t \colon (X_0,d\lambda_0) \hookrightarrow (X,d\lambda_t)$, obtained by pre-composing with a smooth isotopy $\psi_t \colon X_0 \to X_0$ produced by Moser's trick, where $\psi_t$ is the identity near the boundary. In addition, we may assume that $\tilde{\Phi}_t \circ \psi_t$ is exact in the sense that $(\tilde{\Phi}_t \circ \psi_t)^*\lambda_t=\lambda_0+dg_t$. 

Finally, we can use the Liouville flow to extend this isotopy of $X_0$ to an exact symplectic isotopy that is cylindrical outside of a compact subset.
\end{proof}

 For the following results we assume that $\dim M=3$. Given any periodic orbit $\gamma \subset M$ of the Liouville flow on $M$ we can extend it to a cylinder $\gamma \times [-\epsilon,\epsilon] \subset M \times [-\epsilon,\epsilon]$ which is an exact Lagrangian with cylindrical ends for the Liouville form $\eta+s\beta$. Indeed, the pull-back of the Liouville form vanishes along the entire cylinder. From Corollary \ref{cor:completions} we thus get:
 \begin{cor}
 \label{cor:Lagcyl}
 Assume that the Liouville vector field $\zeta_\lambda \in \Gamma(TX)$ of the four-dimensional Liouville domain $(X^4,\lambda)$ is nowhere characteristic on its skeleton, which is the smooth hypersurface $M^3 \subset X^{2n}$, and that $\nu \in \ker \lambda|_M$ is a non-vanishing normal to $M$ for which $\zeta_\lambda$ is $[\nu^*]$-repelling along $M$. Then, any periodic orbit $\gamma$ of the Liouville flow on $M$ can be extended to an exact Lagrangian embedding of $\gamma \times [-1,1]$ with cylindrical ends and a Legendrian boundary that meets both components of the contact boundary $\partial X$.

When $\gamma \subset M$ moreover is contained inside an embedded ball $B \subset M$, then the Legendrian boundary components $\gamma \times \{\pm 1\}$ of the above Lagrangian cylinder are also contained inside balls $B_\pm \subset \partial X$ in the contact boundary of $X$. 
 \end{cor}

Recall the notion from of a so-called \textbf{linear Liouville pair} $(\alpha_+,\alpha_-)$ which is a pair of one-forms $\alpha_\pm \in \Omega^1(M)$ for which $(1+s)\alpha_++(1-s)\alpha_-$ defines the Liouville form on a Liouville domain $M \times [-1,1]_s$. This notion was first considered in \cite{Mitsumatsu}, and further developed in \cite{Hozoori} and \cite{Massoni:Anosov}. The Liouville form above given by $\eta+s\beta$ for $|s| \ge 0$ sufficiently small is of this form; more precisely, it is induced by the contact pair $\alpha_\pm \coloneqq \frac{1}{ 2}(\eta\pm \epsilon \beta) \in \Omega^1(M)$, where $\epsilon>0$ is sufficiently small.
\begin{rmk}A general Liouville pair induces a skeleton that is not necessarily smooth; see Hozoori's recent work \cite{Hozoori:Skeleton}. The notion of a Liouville pair first appeared in work by Mitsumatsu \cite{Mitsumatsu}, who exhibited connections to Anosov flows on $M$. We refer to work by Hozoori \cite{Hozoori} for the full correspondence between Anosov flows on $M$ and a certain sub-class of the Liouville pairs. In related work by Massoni \cite{Massoni:Anosov} connections were also established between a different type of Liouville pairs and Anosov flows.
\end{rmk}

\subsection{Liouville-Hamiltonian structure}
\label{sec:Liouville}
 
Here we propose the definition of an abstract structure on an odd-dimensional manifold that models the symplectic structure near a smooth skeleton of a Liouville domain that is of codimension one; we call this a \textbf{Liouville-Hamiltonian structure}. In view of Theorem \ref{thm:classification}, this structure contains the data needed to recover the Liouville dynamics near a hypersurface $M \subset (X,\lambda)$ in the case when the Liouville vector field $\zeta_\lambda$ is nowhere characteristic along $M$ (and thus in particular tangent to it), up to a Hamiltonian vector field of a Hamiltonian that vanishes to second order along $M$. 
\begin{dfn}
\label{dfn:LHstructure}
A \emph{Liouville-Hamiltonian structure} is a triple $(M^{2n-1},\eta,\beta)$, $\eta,\beta \in \Omega^1(M)$, where
\begin{enumerate}
\item the pair $(M,d\eta)$ is a \emph{Hamiltonian structure}, i.e.~$d\eta$ is maximally non-degenerate;
\item $\ker \eta \supset \ker d\eta$; and
\item $\ker\beta \cap \ker d\eta =\{0\}.$
\end{enumerate}
\end{dfn}

\begin{lma} 
Let $(M,\eta,\beta)$ be a Liouville-Hamiltonian structure. Any one-form $\alpha \in \Omega^1(M)$ which satisfies $\ker \alpha \supset \ker d\eta$ can be represented by $\iota_{\zeta_\alpha}d\eta=\alpha$. If we, moreover, require that $\zeta_\alpha \in \ker\beta$, then this representation is unique, and we write $\zeta_{\alpha,\beta} \in \ker \beta$.
\end{lma}

\begin{dfn}
 To any Liouville-Hamiltonian structure $(M,\eta,\beta)$ there are two canonically associated vector fields:
\begin{enumerate}
\item The \emph{Liouville vector field $\zeta$}, which is given by $\zeta=\zeta_{\eta,\beta} \in \ker \beta$; and
\item The \emph{characteristic vector-field $\mathcal{C}$}, which is determined uniquely by the conditions $\mathcal{C}=\mathcal{C}_{\eta,\beta} \in \ker d \eta$ and $\beta(\mathcal{C}_{\eta,\beta})=1$.
\end{enumerate}
\end{dfn}

\begin{rmk}
 \begin{enumerate}
 \item If, moreover, the relation $\ker d\beta \supset \ker d\eta$ is satisfied, the triple $(M,d\eta,\beta)$ is a so-called \textbf{stable Hamiltonian structure}. However, as Lemma \ref{lma:stabham} below shows, $M$ cannot be closed in this case.
 \item Condition (3) in Definition \ref{dfn:LHstructure} implies that $\beta \wedge (d\eta)^{\wedge(n-1)}$ is a volume form on $M$, in particular $M$ is orientable.
 \item The Liouville vector field $\zeta$ is tangent to $\ker d\eta$ precisely at its critical points; c.f.~the notion of nowhere characteristic Liouville flow in Definition \ref{dfn:nowherechar}.
 \item Whenever Parts (1) and (2) of Definition \ref{dfn:LHstructure} are satisfied for some $\eta \in \Omega^1(M)$, there exist plenty of choices $\beta \in \Omega^1(M)$ that make $(M,\eta,\beta)$ into a Liouville-Hamiltonian structure. Given one such choice $\beta$, all other choices can be written of the form $\beta'=\pm e^g\beta+\iota_{\Xi}d\eta$ with $\Xi \in \ker \beta$ and $e^g \colon M \to \R$ smooth. The Liouville vector fields induced by $\beta$ and $\beta'$ agree if and only if $\Xi \in \ker \eta \cap \ker \beta$.
 \end{enumerate}
\end{rmk}

The first basic examples of Liouville-Hamiltonian structures are constructed by stabilising Liouville domains by an $\R$-factor, and are thus open. Closed examples are less easy to construct; see Subsection \ref{sec:examples} for some classic examples. 
\begin{ex} For $(X^{2(n-1)},\lambda)$ an exact Liouville manifold, we consider the open manifold $M=X^{2(n-1)} \times \R_z$ with $\eta=\pi_X^*\lambda$ and $\beta=dz$. Note that $(M,\eta+\beta)$ is the so-called \textbf{contactisation} of $(X,\lambda)$, which is a contact manifold with contact form $\eta+\beta$.
\end{ex}
However, as the following lemma shows, in the case when $M$ is closed, a Liouville-Hamiltonian structure is never even a stable Hamiltonian structure (which is a property that is stronger than the contact condition for the one-forms $\eta+t\beta$ for all small $t$). The strengthened version of this statement is also highly relevant here, and was pointed out to the author by Hozoori in private communication.
\begin{lma}
\label{lma:stabham}
For a Liouville-Hamiltonian structure $(M,\eta,\beta)$ where $M$ is closed and of dimension $\dim M=2n-1$, the identity $d\beta(\mathcal{C},\zeta_\eta)=n-1$ must hold somewhere. In particular, when $n \ge 2$, the inclusion $\langle \mathcal{C} \rangle=\ker d\eta \subset \ker d\beta$ fails somewhere, i.e.~$(M,d\eta,\beta)$ is not a stable Hamiltonian structure.
\end{lma}
\begin{proof}
Consider the volume form $\Omega=\beta \wedge (d\eta)^{\wedge (n-1)}$ on $M$ and let $\phi^t$ be the Liouville flow generated by $\zeta_\lambda$. Using Cartan's formula
$$\frac{d}{dt}(\phi^t)^*\beta=\iota_{\zeta_\eta}d\beta+d\iota_{\zeta_\eta}\beta=\iota_{\zeta_\eta}d\beta$$
and the identity $(\phi^t)^*d\eta=e^t\,d\eta$, we compute
\begin{eqnarray}
\lefteqn{\left.\frac{d}{dt} (\phi^t)^*\Omega\right|_{t=0}=}\\
&=&\left.\frac{d}{dt}\left((\phi^t)^*\beta \wedge e^{(n-1)t}(d\eta)^{\wedge(n-1)}\right)\right|_{t=0}\\
&=&\left(\iota_{\zeta_\eta}d\beta+(n-1)\beta\right)\wedge (d\eta)^{\wedge (n-1)}.
\end{eqnarray}
Since the total integral of this top-form must vanish, it follows that the form itself must vanish at some point. Vanishing at a point is in turn equivalent to the equality $d\beta(\mathcal{C},\zeta_\eta)=n-1$ being satisfied there.
 \end{proof}
Next we introduce a particularly well-behaved type of Liouville-Hamiltonian structures. Below we will show that they model Liouville vector fields that are repelling along a hypersurface; see Lemma \ref{lma:equiv1} and Proposition \ref{prp:liouville}.
 
\begin{dfn}
\label{dfn:lindef}
 We say that the Liouville-Hamiltonian structure $(M,\eta,\beta)$ is of \emph{linear contact-deformation type} if the path $\eta+t\beta \in \Omega^1(M)$ of one-forms satisfies the property that the top-form
$$ \omega_t \coloneqq (\eta+t\beta) \wedge (d\eta+t\,d\beta)^{\wedge(n-1)} \in \Omega^{2n-1}(M)$$
has a $t$-derivative which at $t=0$ is a volume-form $\left.\frac{d}{dt}\omega_t\right|_{t=0}>0$ that induces the same orientation as $\beta \wedge (d\eta)^{\wedge(n-1)}$.
\end{dfn}
\begin{rmk}
 
\begin{enumerate}
 \item Since $\ker \eta \supset \ker d\eta$ we have
$$ \eta \wedge (d\eta)^{\wedge(n-1)}=0.$$
In other words, the above top-form $\omega_t$ automatically has vanishing 0:th order term in $t$. In particular, being of linear contact-deformation type implies that $\eta+t\beta$ are contact forms whenever $|t| >0$ is small but non-zero.
\item For any $\epsilon>0$ sufficiently small, the pair of one-forms $\alpha_\pm \coloneqq \eta\pm \epsilon \beta$ yields a bi-contact structure as defined by Mitsumatsu \cite{Mitsumatsu} if, in addition, $\eta \neq 0$ is assumed to be non-zero everywhere. Indeed, the orientations induced by the two contact forms $\alpha_\pm$ differ (see e.g. Part (2) of Lemma \ref{lma:equiv1} below), and $\ker \alpha_+ \neq \ker \alpha_-$ holds everywhere.
\item In the case when $\dim M=3$ we have $\eta\wedge d\eta=0$, which means that $\ker \eta$ defines a (possibly singular) two-dimensional foliation on $M$. When $\eta \neq 0$ is everywhere non-vanishing, this foliation is non-singular, and $(M,\eta,\beta)$ being of linear contact-deformation type is equivalent to the following property: the family $\eta+t\eta$ of one-forms gives a \emph{linear deformation (of this plane field) into a positive contact structure} in the sense of the definition by Eliashberg--Thurston in \cite[Chapter 2]{Confoliations}. This is also our motivation for the choice of terminology in Definition \ref{dfn:lindef}.
\end{enumerate}
\end{rmk}
\begin{lma}
\label{lma:equiv1}
 
The following statements are equivalent for a Liouville-Hamiltonian structure $(M,\eta,\beta)$.
\begin{enumerate}
\item $(M,\eta,\beta)$ is of linear contact-deformation type in the sense of Definition \ref{dfn:lindef};
\item $\beta \wedge (d\eta)^{\wedge(n-1)}+(n-1)\eta \wedge d\beta \wedge (d\eta)^{\wedge(n-2)}> 0 $ is a volume form whose orientation agrees with $\beta \wedge (d\eta)^{\wedge(n-1)}$;
\item The inequality $d\beta(\mathcal{C},\zeta_\eta) > -1$ holds everywhere; and
\item The two-form
$$\omega=d(\eta+s\beta)=d\eta+ds \wedge \beta+s\,d\beta \in \Omega^2(M \times \R_s)$$ which is symplectic on $M \times [-\epsilon,\epsilon]$ for $\epsilon>0$ sufficiently small, has an $[ds]$-repelling Liouville vector field for the Liouville form $\eta+s\beta$.
\end{enumerate}
\end{lma}
\begin{proof}
 
(1) $\Leftrightarrow$ (2): The term of lowest order in $t$
\begin{eqnarray*} \lefteqn{(\eta + t\beta) \wedge (d\eta+t\,d\beta)^{\wedge(n-1)}=}\\
&=&t\left(\beta \wedge (d\eta)^{\wedge(n-1)}+(n-1)\eta \wedge d\beta \wedge (d\eta)^{\wedge(n-2)}\right)+\mathbf{O}(t^2)
\end{eqnarray*}
is of order one, and agrees with the expression in (2). The equivalence is thus immediate.

(2) $\Leftrightarrow$ (3): First, consider the contraction of the top form 
\begin{eqnarray*}
 \lefteqn{\iota_{\mathcal{C}}(\beta \wedge (d\eta)^{\wedge(n-1)}+(n-1)\eta \wedge d\beta \wedge (d\eta)^{\wedge(n-2)})=}\\
 &=& (d\eta)^{\wedge(n-1)}-(n-1)\eta\wedge \iota_{\mathcal{C}}d\beta \wedge (d\eta)^{\wedge(n-2)}
 \end{eqnarray*}
with $\mathcal{C}$. Further contracting this form with $\zeta_\eta$ gives us
$$ (n-1)\eta \wedge (d\eta)^{\wedge(n-2)}-(n-1)d\beta(\zeta_\eta,\mathcal{C}) \cdot \eta\wedge (d\eta)^{\wedge(n-2)},$$
since $\iota_{\zeta_\eta}d\eta =\eta$. We thus see that $d\beta(\zeta_\eta,\mathcal{C}) < 1$ is equivalent to the expression in (2) being a volume form of the correct orientation.

(3) $\Leftrightarrow$ (4): The Liouville vector field of $\lambda=\eta+s\beta$ can be written as
$$ \zeta_\lambda=\zeta_\eta+s(f\partial_s+\Xi+g\mathcal{C})$$
where $V=f\partial_s+\Xi+g\mathcal{C}$ satisfies $\iota_{sV} d\lambda+\iota_{\zeta_\eta}s\,d\beta=s\beta$, and $f,g \colon M \to \R$, with $\Xi \in \ker \beta \subset \Gamma(TM)$.

Since $d\lambda=d\eta+ds\wedge \beta+s\,d\beta$ we get
$$s\beta=\iota_{sV} d\lambda+\iota_{\zeta_\eta}s\,d\beta=s(f\beta+\iota_\Xi d\eta-g ds+s\iota_\Xi d\beta+sg\iota_\mathcal{C}d\beta)+\iota_{\zeta_\eta}s\,d\beta$$
and immediately conclude, by comparing orders of $s$, that $g=0$ and consequently $\Xi \in \ker d\beta$. In other words
$$\beta=f\beta+\iota_\Xi d\eta+\iota_{\zeta_\eta}d\beta$$
where $\Xi \in \ker \beta \cap \ker d\beta$.

The property of the Liouville vector field $\zeta_\lambda$ being $[ds]$-repelling is equivalent to $f>0$ being positive, which is equivalent to $\iota_{\zeta_\eta}d\beta=h\beta-\iota_\Xi d\eta$ where $h<1$. This gives the sought equivalence.
\end{proof}
 
There are Liouville-Hamiltonian structures $(M,\eta,\beta)$ which are of linear contact-deformation type of a particularly strong form, as described in the following lemma.
\begin{lma}
The following statements are equivalent for a Liouville-Hamiltonian structure $(M,\eta,\beta)$.
\begin{enumerate}
 \item The one-form $\beta$ satisfies $d\beta(\mathcal{C},\zeta_\eta)>0$ (in particular $\zeta_\eta$ is non-zero); 
 \item The following top-form on $M$ is non-vanishing 
 $$\eta \wedge d\beta \wedge (d\eta)^{\wedge(n-2)}>0$$
 and induces the same orientation as $\beta \wedge (d\eta)^{\wedge(n-1)}>0$
 \end{enumerate}
In either case, $(M,\eta,\beta)$ is of linear contact-deformation type, and has the additional property that the Reeb vector field of the contact form $\eta+t\beta \in \Omega^1(M)$ is nowhere contained inside $\langle \mathcal{C},\zeta_\eta\rangle \subset TM$ for all small $t \neq 0$.
\end{lma}
\begin{proof}
(1) $\Leftrightarrow$ (2): This follows from the same computation as in the proof of the equivalence (2) $\Leftrightarrow$ (3) in Lemma \ref{lma:equiv1}.

The last consequence follows since the Reeb vector field is in $\ker (d\eta+t d\beta)$, where (1) implies that $\langle \mathcal{C},\zeta_\eta\rangle \cap \ker (d\eta+t d\beta )=\{0\}$ when $t \neq 0$; to that end, recall that $ \ker d\eta = \R\mathcal{C} \subset \Gamma(TM)$.
\end{proof}
The following statement is an easy consequence of the main theorem of Zung \cite[Theorem 1]{Zung}, which contains the non-trivial part of the argument.
\begin{prp}
\label{prp:hypertaut}
Assume that $(M^3,\eta,\beta)$ is a Hamiltonian-Liouville structure with $\dim M =3$ and $\eta \neq 0$ is satisfied everywhere. The inequality $d\beta(\mathcal{C}_{\eta,\beta},\zeta_{\eta,\beta})>0$ implies that the foliation $\mathcal{F}$ defined by $T\mathcal{F}=\ker \eta$ has no invariant measure. Conversely, if $\mathcal{F}$ has no invariant measure, then we can find some $\beta'=e^g\beta+\iota_\Xi d\eta \in \Omega^1(M)$ for which $d\beta(\mathcal{C}_{\eta,\beta'},\zeta_{\eta,\beta'})>0$.
\end{prp}
\begin{proof}
By the proof of \cite[Theorem 1]{Zung}, also see \cite[Proposition 4.5]{Massoni:Taut}, the smooth foliation $\mathcal{F}$ has no invariant measure if and only if there exists some $\tilde{\beta} \in \Omega^1(M)$ which satisfies
\begin{equation}
 \label{eq:contact}
\eta \wedge d\tilde{\beta}>0 \:\:\text{and}\:\: \tilde{\beta}\wedge d \eta \ge 0
\end{equation}
where we orient $M$ via the volume form $\beta \wedge d\eta>0$.

If $d\beta(\mathcal{C}_{\eta,\beta},\zeta_{\eta,\beta})>0$, then the first inequality in \eqref{eq:contact} can be seen to be satisfied for $\tilde{\beta}=\beta$, while the second inequality is even strict.

Conversely, assuming the existence of $\tilde{\beta}$ that satisfies the above Inequalities \eqref{eq:contact}, we get a Liouville-Hamiltonian structure $(M,\eta,\beta'=\tilde{\beta}+\epsilon\beta)$. When is $\epsilon>0$ sufficiently small the inequality $d\beta'(\mathcal{C}_{\eta,\beta},\zeta_{\eta,\beta})>0$ is still satisfied.
\end{proof}
Liouville manifolds whose skeleta admit Liouville-Hamiltonian structures of the type described by Proposition \ref{prp:hypertaut} were investigated by Massoni in \cite{Massoni:Taut}. 
We refer to that article as well as \cite{Hozoori:Skeleton} for the Anosov properties of the corresponding Liouville flow on $M$.

\subsection{Hamiltonian-Liouville structures on hypersurfaces of Liouville domains}

For any Liouville-Hamiltonian structure $(M,\eta,\beta)$, we can consider the one-form $\lambda=\eta+s\beta$ on $M \times \R_s$, where $\eta,\beta \in \Omega^1(M)$ are extended by pull-back under the canonical projection to $M$. One easily verifies the following, where the equivalence in the second statement follows from Lemma \ref{lma:equiv1} combined with Part (2) of Theorem \ref{thm:classification}.
\begin{prp}
\label{prp:liouville}
For any Liouville-Hamiltonian structure $(M,\eta,\beta)$ and $\epsilon>0$ sufficiently small, we get a Liouville manifold $(M \times [-\epsilon,\epsilon]_s,\lambda=\eta+s\beta)$ whose Liouville-vector field $\zeta_\lambda$ is nowhere characteristic on $M$, and which coincides with $\zeta_\lambda|_M=\zeta_{\eta,\beta}$ along $M$.

Moreover, $(M,\eta,\beta)$ is of linear contact-deformation type if and only if $\zeta_\lambda$ is $[ds]$-repelling along $M$.
\end{prp}

Proposition \ref{prp:beta} above can be seen as a converse to the above proposition. Namely, when the Liouville vector field $\zeta_\lambda \in \Gamma(TX)$ of $(X,\lambda)$ is nowhere characteristic on the hypersurface $M \subset X$, and we take $\nu \in \ker(\lambda_M)$ to be a choice of non-vanishing normal vector field along $M$, then $(M,\eta=\lambda|_{TM},\beta=\beta_\nu)$ is a Liouville-Hamiltonian structure where $\beta_\nu \in \Omega^1(M)$ is the one-form provided by Proposition \ref{prp:beta}.

Moreover, Theorem \ref{thm:classification} implies that there is an embedding $\iota \colon M \times [-\epsilon,\epsilon]_s \hookrightarrow X$, where $M=\{s=0\}$ and $\partial_s|_M=\nu$, and for which $$\iota^*\lambda=\eta+s\beta_\nu+d(s^2G)$$
is satisfied for some smooth function $G$ defined near $M$. In other words, the induced Liouville-Hamiltonian structure on a hypersurface $M \subset (X,\lambda)$ recovers the Liouville form on $X$ up to the exact one-form $d(s^2G)$.

Recall that $\beta_\nu$ has the following dependence on the choice of normal vector-field $\nu$. For any different choice of normal $\nu'=e^g\nu+\Xi$ with $\Xi \in \ker \lambda \cap \Gamma(TM)$, Proposition \ref{prp:beta} implies that
$$\beta_{\nu'}=e^g\beta+\iota_\Xi d\eta,\:\: \Xi \in \ker \eta \subset \Gamma(TM).$$
Note that the two Liouville-Hamiltonian structures $(M,\eta=\lambda|_{TM},\beta_\nu)$ and $(M,\eta=\lambda|_{TM},\beta_{\nu'})$ induce the same Liouville vector-fields along $M$.

If we instead consider a general Liouville-Hamiltonian structure $(M,\eta,\beta')$ with $\beta'=e^g\beta+\iota_\Xi d\eta$ where $\Xi \in \Gamma(TM)$ but not necessarily $\Xi \in \ker \eta$, then the new Liouville vector field on $M$ is equal to $\zeta_\eta+F\mathcal{C}_{\eta,\beta}$, where $F=e^{-g}\eta(\Xi).$ This Liouville-Hamiltonian structure is induced by the Liouville form 
$$\lambda' = \eta+s\beta-d(sF) \in \Omega^1(M \times [-\epsilon,\epsilon]).$$
via the vector-field $\nu'=e^g\partial_s+\Xi \in \ker(\lambda'|_M)$ normal to $M$.

\subsection{Closed examples: McDuff and torus bundle domains}
\label{sec:examples}

Constructing examples of Liouville-Hamiltonian structures on closed manifolds $M^{2n-1}$ is a highly non-trivial task. Here we present the two most well-studied examples, which both are three-dimensional. More three-dimensional examples can be found using the connections with Anosov flows established in \cite{Hozoori} and \cite{Massoni:Anosov}.

\subsubsection{McDuff domains}
\label{sec:McDuff}

The McDuff domain were first constructed by McDuff in \cite{McDuff}. Consider a unit cotangent bundle $ \pi \colon U^*\Sigma \to \Sigma$ of an oriented surface with the contact form $\alpha_g=p\,dq|_{T(U^*\Sigma)}$ induced by a Riemannian metric $g$ on $\Sigma$. Further, let $\Theta_g \in \Omega^1(U^*\Sigma)$ be the connection one-form on the unit \emph{cotangent} bundle induced by the same metric $g$. We have three canonical vector-fields on $T(U^*\Sigma)$:
\begin{itemize}
 \item The angular vector-field $\partial_\theta$ in the fibre induced by the metric;
 \item The Reeb vector-field $R_g$, i.e.~the vector field that generates the cogeodesic flow; and
 \item The unit vector in the intersection $h \in \ker\alpha \cap \ker \Theta_g $ for which $\langle R_g,h,\partial_\theta\rangle$ is a basis positively oriented by $\alpha_g \wedge d\alpha_g$, which means that $\langle h, R_g,\partial_\theta \rangle$ is a positively oriented basis for the orientation of $T^*\Sigma$ induced by the locally defined oriented basis $\langle D\pi(h), D\pi(R_g) \rangle$ of $T\Sigma$.
\end{itemize}
Standard computations give us
\begin{itemize}
\item $\alpha_g(\partial_\theta)=0=\Theta_g(R_g)$;
\item $\alpha_g(R_g)=1=\Theta_g(\partial_\theta)$;
\item $d\alpha_g(R_g,\cdot)=0=d\Theta_g(\partial_\theta,\cdot)$; and
\item $d\alpha_g(\partial_\theta,h)=1$;
\end{itemize}
Under the additional assumption that $d\Theta_g( h,R_g )=-1$, i.e.~$g$ is a hyperbolic metric of constant scalar curvature, it follows that $d(\alpha_g-\Theta_g)(h,\cdot)=\alpha_g-\Theta_g$. In other words, the triple $(U^*\Sigma, \alpha_g - \Theta_g, \alpha_g)$ is a Liouville-Hamiltonian structure with Liouville vector field given by $ \zeta =h$, which generates the so-called horocycle flow. The characteristic vector field is given by $\mathcal{C}=R_g+\partial_\theta$ (this is a cogeodesic flow with a magnetic term) and hence $d\beta(\mathcal{C},\zeta)=d\alpha_g(R_g+\partial_\theta,h)\equiv 1$. This means that the Liouville-Hamiltonian structure is of the linear contact-deformation type in the strong sense as described in Proposition \ref{prp:hypertaut}. 

Since $U^*\Sigma \to \Sigma$ is a trivial $S^1$-bundle when $\Sigma$ is open, in the case of constant curvature $-1$ we have $d\Theta_g=-\pi^*\sigma=-d\pi^*\gamma$ where $\sigma \in \Omega^2(\Sigma)$ is the area form on $\Sigma$ and $\gamma \in \Omega^1(\Sigma)$ is a choice of primitive (which exists when $\Sigma$ is open). Thus, there is an embedding
\begin{gather*}
(U^*\Sigma,\alpha_g-\Theta_g) \hookrightarrow (T^*\Sigma,p\,dq),\\
x \mapsto x+\Gamma_{ \gamma},
\end{gather*}
induced by fibre-wise addition of the section $\Gamma_{ \gamma} \subset T^*\Sigma$ of $ \gamma \in \Omega^1(\Sigma)$ that preserves the exterior differentials of the one-forms.

We refer to \cite{Cieliebak:Anasov} as well as Subsection \ref{sec:McDuff2} below for more details concerning the symplectic topology of the McDuff domain itself.

\subsubsection{Torus bundle domains}
\label{sec:torus}
Liouville structures on torus bundle domains were first constructed by Mitsumatsu in \cite{Mitsumatsu}. Consider a matrix $A \in \OP{SL}_2(\Z)$ which is hyperbolic, i.e.~with one eigenvalue $e^\nu$ with eigenvector $\mathbf{v}$ and one eigenvalue $e^{-\nu}$ with eigenvector $\mathbf{w}$, where $\nu \neq 0$. The matrix acts on $(T^*\T^2=\T^2_\theta \times \R^2_p, \mathbf{p}\,d\boldsymbol{\theta} )$ by the exact symplectomorphism $
\Phi(\boldsymbol{\theta},\mathbf{p}) \coloneqq ( (A^{tr})^{-1} (\boldsymbol{\theta}), A (\mathbf{p}))$
and induces an action on the subset
$$ \tilde{V} \coloneqq { \T^2} \times (\R_{> 0}\mathbf{v}+\R\mathbf{w}) \subset T^*\T^2.$$
The properly embedded hypersurface
$$\tilde{M} \coloneqq \T^2 \times (\R_{>0}\mathbf{v}+ 0\mathbf{w} ) \subset \tilde{V} \subset T^*\T^2$$
is tangent to the Liouville flow and fixed set-wise by the symplectomorphism $\Phi$. The quotient of $\tilde{V}$ by the group $\Z=\langle\Phi\rangle$ produces a complete Liouville manifold $M \times \R$ whose skeleton is given by the quotient $M$ of $\tilde{M}$, where $M$ thus is a $\T^2$-bundle over $S^1$ twisted by $A$.

 The Liouville structure on $T^*\T^2$ induces a Liouville-Hamiltonian structure
$$(\tilde{M},\eta,\beta)=(\T^2 \times (\R_{>0})_t,t\mathbf{v}d\boldsymbol{\theta},t^{-1}\mathbf{w}d\boldsymbol{\theta}),$$
where $\Phi$ takes the form $\Phi(\boldsymbol{\theta},t)=((A^{tr})^{-1}(\boldsymbol{\theta}),e^\nu t).$
The Liouville vector field is given by $\zeta=t\partial_t$, while the characteristic vector-field is $\mathcal{C}=t\mathbf{w}^*\partial_{\boldsymbol{\theta}}$ where $\mathbf{w}^*\bullet \mathbf{w}=1$ and $\mathbf{w}^*\bullet \mathbf{v}=0$ and thus $(A^{ tr})^{-1}(\mathbf{w}^*)=e^\nu\mathbf{w}^*$; additionally $(A^{ tr})^{-1}(\mathbf{v}^*)=e^{-\nu}\mathbf{v}^*$ is satisfied, where $\mathbf{v}^*\bullet \mathbf{w}=0$ and $\mathbf{v}^* \bullet \mathbf{v}=1$. Note that both $\eta$ and $\beta$ thus are invariant under $\Phi$.

Also in this case the Liouville-Hamiltonian structure is of linear contact-deformation type of the stronger form as described by Proposition \ref{prp:hypertaut}, since
$$d(t^{-1}\mathbf{w}d\boldsymbol{\theta})(\mathcal{C},\zeta)=1$$
is satisfied everywhere.

\section{Classifying weakly exact Lagrangians (Theorem \ref{thm:a})}

 The result that we want to establish is a topological classification of weakly exact Lagrangians inside the McDuff and torus bundle domains. The strategy of the proof is to pass to certain infinite covering spaces of the Liouville domain $(\tilde{V},\lambda) \to (V,\lambda)$ whose symplectic structures are better understood than the original space. The reason for why infinite covers are simpler is that the universal cover of $V$ admits an embedding into the symplectic vector space. The most important step of this strategy is produce a covering under which some given weakly exact Lagrangian can be lifted. Once this has been done, we can use the following elementary result which states that lifts of weakly exact Lagrangians under covering spaces still are weakly exact.
\begin{lma}
\label{lma:weaklyexact}
Let $p \colon (\tilde{X},\tilde\omega) \to (X,\omega)$ be a symplectic covering, and $\iota \colon L \hookrightarrow (X,\omega)$ be a weakly exact Lagrangian. If $\iota$ admits a lift $\tilde\iota \colon L \hookrightarrow (\tilde{X},\tilde\omega)$, i.e. $p\circ\tilde{\iota}=\iota$, then $\tilde\iota$ is a weakly exact Lagrangian embedding as well.
\end{lma}
\begin{proof}
Any element in $\pi_2(\tilde{X},\tilde\iota(L))$ has the same symplectic area as the corresponding image in $\pi_2(X,\iota(L)).$
\end{proof}
A crucial non-existence result for weakly exact Lagrangians that we rely on is the following.
\begin{thm}[Lalonde--Sikorav \cite{LalondeSikorav:SousVarietees}]
\label{thm:lalondesikorav}
There are no closed weakly exact Lagrangian submanifolds in $(T^*\Sigma,p\,dq)$ when $\Sigma$ is an open manifold. 
\end{thm}
\begin{proof}
 We argue that the Floer homology groups for weakly exact Lagrangians in $T^*\Sigma$ are well defined, invariant under Hamiltonian isotopies, and non-zero. We refer to e.g.~\cite{FOOO:I} for the definition of Floer homology, which is well-defined when the Lagrangians are weakly exact and the symplectic manifold is well-behaved at infinity. 
 
 In order to ensure that the relevant pseudoholomorphic curves remain confined to a priori given compact subsets of $T^*\Sigma$, which is needed to have a Floer complex that is well-defined and invariant, it suffices to produce a tame almost complex structure for which $T^*\Sigma$ is convex at infinity in the sense of Gromov \cite{Gromov:Pseudo}. Indeed, one can readily see that $T^*\Sigma$ can be exhausted by a family $W_1 \subset W_2 \subset W_3 \ldots \subset T^*\Sigma$ of subcritical Weinstein subdomains; any tame almost complex $J$ structure for which all boundaries $\partial W_i$ are $J$-convex is then of the required type.
 
 For an open manifold $\Sigma$ one can find a smooth function $f$ without critical points. One way to construct this function is as follows. Let $C$ be a compact cobordism whose boundary $\partial C =B_1 \sqcup B_2$ admits a decomposition into connected components where $B_1 \neq \emptyset$. Any smooth function defined near $B_2$ (which is allowed to be the empty set) that has no critical points, can be extended to a smooth function $f\colon C \to \R$ without critical points. To see this, one can start by extending the function to $\tilde{f} \colon C \to \R$ with only a finite number of Morse critical points that all are contained in $C \setminus \partial C$. We can then use an appropriate diffeomorphism $\iota \colon C \hookrightarrow C$ that is the identity near $B_2$, and which maps $C$ into the complement of the critical points, and use $f\coloneqq \tilde{f}\circ \iota$ as our sought function without critical points. (The diffeomorphism $\iota$ can be constructed by e.g.~using a smooth arc that connects a point in the boundary $B_1$ with all the critical points of $\tilde{f}$, and then performing an appropriate isotopy supported in a neighbourhood of this arc.) Since $\Sigma$ admits a proper Morse function, we can decompose it into a possibly infinite number of compact cobordisms, and argue by induction to construct the sought function $f$ without critical points on all of $\Sigma$.
 
 One we have managed to construct the function $f$ without critical points, fibre-wise addition with the family of sections $t\cdot df \in \Gamma(T^*M)$ induce a Hamiltonian isotopy which, due to the non-vanishing of $df$, displaces any compact subset from itself when $t \gg 0$. This contradicts the fact that Floer homology is well-defined, non-zero, and invariant, for weakly exact closed Lagrangians inside symplectic manifolds that are convex at infinity. 
\end{proof}

The classification problem of Lagrangian submanifolds up to Hamiltonian isotopy is in general wide open. However, in a few number of certain well-behaved four-dimensional symplectic manifolds, there are classification results for weakly exact Lagrangians. The following result by the author will be needed, which is a slightly strengthening of the classification result from \cite{Dimitroglou:Isotopy} by the author joint with Goodman and Ivrii.
\begin{thm}[Theorem B \cite{Dimitroglou:Whitney}]
\label{thm:HamiltonianIsotopy}
Any weakly exact Lagrangian
$$L \subset \left(T^*\T^2=\T^2_{\boldsymbol{\theta}} \times \R^2_{\mathbf{p}},\sum_i dp_i \wedge d\theta_i\right)$$
is Hamiltonian isotopic to a torus $\T^2 \times \{(p_0,p_1)\}$. Under the additional assumption that $L \subset \T^2 \times \Omega$ for some convex subset $\Omega \subset \R^2,$ the Hamiltonian isotopy can, moreover, be assumed to be confined to $\T^2 \times \Omega$.
\end{thm}

\subsection{General restrictions on Lagrangian embeddings}

The Lagrangian adjunction-formula implies that the self-intersection number $L \bullet L \in \Z$ for any orientable Lagrangian surface $L^2 \subset (X^4,\omega)$ satisfies $L \bullet L=-\chi(L)$; see e.g.~\cite{Audin}. In the non-orientable case, the same formula is true modulo two. When $X$ is compact with possibly empty boundary, Poincar\'{e} duality $PD \colon H_*(X) \xrightarrow{\cong} H^{ 4-*}(X,\partial X)$ implies that the square of the cup-product
\begin{gather*} S \colon H^2(X,\partial X) \to H^4(X,\partial X),\\
\alpha \mapsto \alpha \smile \alpha
\end{gather*}
satisfies $S(PD[L])=-\chi(L)\cdot PD([\pt])$. The operation has a lift to the so-called Pontryagin square
$$ P_2 \colon H^2(X,\partial X;\Z_2) \to H^4(X,\partial X;\Z_4),$$
that simply is equal to the modulo four reduction of $S$ on the image of the reduction $H^2(X,\partial X;\Z) \to H^2(X,\partial X;\Z_2)$. Audin shown in \cite[Proof of Proposition 1.3.1]{Audin} that
$$P_2(PD[L])=-\chi(L)\cdot PD([\pt]) \in \Z_4\cdot PD([\pt])$$
is satisfied for any Lagrangian embedding, i.e.~the Pontryagin square detects the Euler characteristic of $L$ modulo four.

Recall that the Euler characteristic of an non-orientable surface is equal to 
$$\chi(\underbrace{\RP^2 \sharp \ldots \sharp \RP^2}_k)=2-k.$$
From this it is easy to derive the following.
\begin{lma}
 All closed Lagrangians in $V$ are either tori, Klein bottles, or more general connected sums of $2k$ number of $\RP^2$'s for $k\ge 1$
\end{lma}
\begin{proof}
Since $(V,\partial)=M \times (I,\partial I)$ is a stabilisation of a three-dimensional manifold, the square of the cup-product $H^2(V,\partial V) \to H^4(V,\partial V)$ vanishes for all choices of coefficients. It follows that the Euler characteristic of any oriented closed Lagrangian must vanish, while it must vanish modulo two in the non-orientable case.
\end{proof}
 Below we will exclude a Lagrangian embedding of the Klein bottle into $V$. The main mechanism for excluding such an embedding is a result due to Shevchishin \cite{LagKlein}, who showed that a Klein bottle in a uniruled symplectic manifold cannot be null-homologous modulo two. The important consequence of this result that we will need is:
\begin{prp}
\label{prp:klein}
There exists no Lagrangian embedding of a Klein bottle in $(T^*\Sigma,p\,dq)$ for any cotangent bundle of a (possibly open) surface $\Sigma$, unless it is a Klein bottle itself, i.e.~$\Sigma=\RP^2 \sharp \RP^2$.
\end{prp}
\begin{proof}
In the case when either $\Sigma$ is an open surface, $\Sigma=\T^2$, or $\Sigma=\RP^2$, then any compact subset of $T^*\Sigma$ admits a conformal symplectic embedding into $ (\CP^2,\omega_{FS})$. To see this we can use the fact that any compact subset of $\Sigma$ admits a Lagrangian embedding into $\C^2$ when $\Sigma$ is open or equal to $\T^2$. Weinstein's Lagrangian neighbourhood theorem can then be used to embed any compact subset of $T^*\Sigma$ into $\C^2$. Note that rescaling of $\C^2$ is a conformal symplectomorphism. Since $\Sigma=\RP^2$ admits a Lagrangian embedding into the standard symplectic projective plane $(\CP^2,\omega_{FS})$, any compact subset of $T^*\RP^2$ thus admits a conformal symplectic embedding into $(\CP^2,\omega_{\OP{FS}})$ as well. Since neither $\C^2$ nor $\CP^2$ admit any Lagrangian embeddings of the Klein bottle by \cite{LagKlein}, the same is also true for $T^*\Sigma$ in these cases.

In the case when $\Sigma$ is closed and satisfies $\rank H_1(\Sigma) \ge 2$, we claim that any embedding of a Klein bottle in $T^*\Sigma$ lifts to some appropriate infinite covering space $T^*\tilde{\Sigma} \to T^*\Sigma$; hence, such an embedding can never be Lagrangian by case treated above. Indeed, any continuous map $f \colon \RP^2 \sharp \RP^2 \to \Sigma$ induces a map $f_* \colon H_1(\RP^2 \sharp \RP^2)=\Z_2 \times \Z \to H_1(\Sigma)$ with $\rank(H_1(\Sigma)/\im f_*) \ge 1$. This means that the image of $f_* \colon \pi_1(\RP^2 \sharp \RP^2) \to \pi_1(\Sigma)$ is a subgroup of infinite index. From this it follows that $f$ lifts to some suitable infinite covering space $\tilde{\Sigma}\to\Sigma$.

What remains is to exclude a Lagrangian embedding of a Klein bottle in $T^*\Sigma$ when $\Sigma=S^2$. In that case, any compact subset of $T^*S^2$ admits a symplectic embedding in $(\CP^1 \times \CP^1,\omega_{FS} \oplus \omega_{FS})$, such that $[0_{S^2}]$ becomes the anti-diagonal class $[0_{S^2}]=[\CP^1]\oplus-[\CP^1]$. Any Lagrangian Klein bottle in $T^*S^2$ thus gives rise to a Lagrangian Klein bottle $L \subset \CP^1 \times \CP^1$ which, by \cite{LagKlein} must live in the anti-diagonal class $[\CP^1]\oplus-[\CP^1] \in H_2(\CP^1 \times \CP^1;\Z_2)$ modulo two. This contradicts the fact that $P_2(PD[L])=-\chi(L) \mod 4=0$ vanishes by Audin's result while, since $[0_{S^2}] \in H_2(T^*S^2;\Z_2)$ admits a lift $[0_{S^2}]_\Z \in H_2(T^*S^2)$, we have $P_2(PD[L])=PD([0_{S^2}]_\Z) \smile PD([0_{S^2}]_\Z) \mod 4=-2 \cdot PD[\OP{pt}]$.
\end{proof}

\begin{rmk}
 Computations of the Pontryagin square can also be used to rule out Lagrangian embeddings of $2k$-fold connected sums of $\RP^2$'s when $2-2k \neq 0 \mod 4$ for certain McDuff and torus-bundle domains.
\end{rmk}
 In the following, we restrict our attention to the case of Lagrangian embeddings $L \subset (V,\lambda)$ that are either tori or Klein bottles. One can rule out the case of $L$ being a Klein bottle. However, we did not manage to rule out the existence of weakly exact $2+2k$-fold connected sums of $\RP^2$'s for general $k \ge 1$.

\subsection{Part (1): Torus bundle domains}

The following crucial result was established in the proof of \cite[Theorem 4]{Cieliebak:Anasov}, which gives important restrictions on the homotopy class of an exact Lagrangian torus or Klein bottle in a torus-bundle domain. 
\begin{prp}[\cite{Cieliebak:Anasov}]
\label{prp:lift}
Any closed Lagrangian submanifold $\iota \colon L \hookrightarrow V$ which is either a torus or a Klein bottle lifts to a cover of $V$ which is symplectomorphic to either $(T^*(S^1 \times \R),p\,dq)$ or $(T^*\T^2,p\,dq)$.
\end{prp}
\begin{proof}
 It was shown in \cite[Lemma 4.7]{Cieliebak:Anasov} that, when $L$ is either a torus or a Klein bottle, then the fundamental group of $L$ is mapped into the image of a torus fibre in $\pi_1(V)$, or it is a cyclic subgroup of $\pi_1(V)$.

In the case when $\pi_1(L)$ is mapped into the kernel $\ker p_* \subset \pi_1(V)$ induced by the fibration $p \colon V \to S^1$, we can clearly lift $L$ to the cover $\tilde{V}\to V$ induced by the subgroup $\ker p_* \subset \pi_1(V)$, i.e.~the symplectic manifold $T^*\T^2$.

In the case when the image of $\pi_1(L)$ is contained inside a cyclic subgroup $\langle g \rangle \subset \pi_1(V)$, but not contained inside $\ker p_* \subset \pi_1(V)$, then \cite[Lemma 4.8]{Cieliebak:Anasov} implies that $L$ lifts to a cover $\tilde{V} \to V$ which is symplectomorphic to $T^*(S^1 \times \R)$.
\end{proof}
Proposition \ref{prp:klein} immediately shows that $L$ cannot be a Klein bottle.

Now consider the case when $L$, in addition, is assumed to be weakly exact. In this case, of the two alternatives for the possible lifts $\tilde{L} \subset \tilde{V} \to V$ provided by Proposition \ref{prp:lift}, only $\tilde{V}=\T^2\times (\R_{> 0}\mathbf{v}+\R\mathbf{w}) \subset T^*\T^2$ is possible; this follows immediately from Theorem \ref{thm:lalondesikorav}, which excludes weakly exact lifts to cotangent bundles of open manifolds, together with the fact that lifts of weakly exact Lagrangians again are weakly exact by Lemma \ref{lma:weaklyexact}.

The sought Hamiltonian isotopy from the lift of the torus to a standard fibre in $T^*\T^2$ then exists by Theorem \ref{thm:HamiltonianIsotopy}.

Since the latter Hamiltonian isotopy of the torus has a compact image, it is disjoint from its image of the group $\Z$ of decks transformations for which $\tilde{V}/\Z=V$, except for possibly finitely many group elements. Hence, we can take a quotient by a finite-index subgroup $m\Z \subset \Z$ which is a finite cover of $V$ in which the lifted torus is Hamiltonian isotopic to a standard fibre.

\begin{rmk}
The question of whether there exists a Hamiltonian isotopy from $L \subset V$ to a Lagrangian torus fibre is harder, since we would need to make the isotopy inside $T^*\T^2$ equivariant with respect to the bundle projection.
\end{rmk}

\subsection{Part (2): McDuff domains}
\label{sec:McDuff2}

Let $\Sigma_g$ denote the closed surface of genus $g$. The symplectic structure on the McDuff-domains can be constructed as
$$ (V,\lambda)=(D^*\Sigma_g \setminus \mathcal{O}_{0_{\Sigma}},\lambda_{can}+\eta),\:\: g \ge 2,$$
where $\mathcal{O}_{0_{\Sigma}}$ is a suitable, arbitrarily small, open tubular neighbourhood of the zero-section,
and $\eta \in \Omega^1(T^*\Sigma_g \setminus \mathcal{O}_{0_{\Sigma}})$ satisfies $d\eta=\pi^*\sigma$ with $\sigma \in \Omega^2(\Sigma_g)$ an area form. Of course, we need to choose $\eta$ so that the Liouville vector field becomes outwards pointing at both boundary components; this is where we must require the genus to satisfy $g \ge 2$. We refer to \cite{Cieliebak:Anasov} for more details.

When passing to an infinite cover
$$(\tilde{V},d\tilde{\lambda}) =(D^*\tilde{\Sigma}_g \setminus \mathcal{O}_{0_{\tilde{\Sigma}}},\lambda_{can}+\tilde{\eta})$$
induced by an infinite cover $\tilde{\Sigma}_g \to \Sigma_g$, we have $\pi^*\sigma=d\pi^*\beta$ for some one-form $\beta \in \Omega^1(\tilde{\Sigma}_g)$. In other words, there is a symplectomorphism
$$ (\tilde{V},d\tilde{\lambda})=(T^*\tilde{\Sigma}_g \setminus \mathcal{O}_{0_{\tilde{\Sigma}}},d\lambda_{can}+d\pi^*\beta) \xrightarrow{\cong} (T^*\tilde{\Sigma}_g \setminus \mathcal{O}_\beta,d\lambda_{can}),\:\: \mathcal{O}_\beta=\mathcal{O}_0+\beta, $$
induced by fibre-wise addition of the section $\beta \in \Omega^1(\tilde{\Sigma}_g)$. Note that $\mathcal{O}_\beta$ is a small open neighbourhood of the section $\beta \in \Omega^1(\tilde{\Sigma}_g)$.

The crucial topological restriction that we need for weakly exact Lagrangian tori inside McDuff domains is the following:
\begin{lma}
\label{lma:klein}
The image of the map of fundamental groups
$$\pi_1(L) \xrightarrow{f_*} \pi_1(\Sigma_g)$$
induced by a continuous map $f \colon L \to \Sigma_g$ with $L$ either a torus or a Klein bottle, and $g \ge 2$, is either trivial or isomorphic to $\Z$.
\end{lma}
\begin{proof}
We start to investigate the map in homology
$$ \phi \colon H_1(L) \to H_1(\Sigma_g)=\Z^{2g}$$
induced by $f$. Since $g \ge 2$ and $H_1(\T^2)=\Z^2$ while $H_1(\RP^2 \sharp \RP^2)=\Z_2 \times \Z$ we conclude that the image has rank at most two or one, respectively. In fact, in the case $L=\T^2$ the rank is also one, since otherwise we could conclude that $f^* \colon H^*(\Sigma_g) \to H^*(L)$ is surjective (i.e.~that $\deg f \neq 0$) which contradicts the fact that $f^*$ is a morphism of unital rings.

It follows that the image $G \subset\pi_1(\Sigma_g)$ of $\pi_1(L) \to \pi_1(\Sigma_g)$ is of infinite index. Indeed, under the surjective quotient $\pi_1(\Sigma_g)\to H_1(\Sigma_g)$ the image of $G$ is mapped to the subgroup $\phi(H_1(L)) \subset H_1(\Sigma_g)$ which is of infinite index.

It then follows from \cite[Theorem 1]{Jaco} that the image $G$ is a free subgroup, and hence it is either trivial or isomorphic to $\Z$.
\end{proof}

\begin{lma}
\label{lma:lift}
Consider a section $\Gamma_\beta \subset T^*\tilde{\Sigma}$ where $\beta \in \Omega^1(\tilde{\Sigma})$ is a primitive of an area form on the open surface $\tilde{\Sigma}=\R^2$ or $\R \times S^1$, i.e.~the infinite cover of $\Sigma_g$ that was considered above. The symplectic manifold $(T^*\tilde{\Sigma} \setminus \Gamma_\beta,d\lambda_{can})$ is symplectomorphic to
\begin{enumerate}
\item $(\C^* \times \C,\omega_0)$ when $\tilde{\Sigma}=\R^2$;
\item $(T^*S^1 \times \C^*,d(p\,d\theta)\oplus \omega_0)$ when $\tilde{\Sigma}=\R \times S^1$.
\end{enumerate}
\end{lma}
\begin{proof}
Case (1): There exists a change of coordinates on $\tilde{\Sigma}$ that makes the area form $d\beta$ into a linear area form on $\tilde{\Sigma}=\R^2_{q_1,q_2}$. After adding an exact one-form to $\beta$ (this is induced by a global symplectomorphism of the tangent bundle), we may thus assume that $\beta=q_1\,dq_2$. 

The complement of the section $\Gamma_{q_1\,dq_2}$ can be seen to be isomorphic to $\C^* \times \C$ in the following manner. First, there is a canonical symplectomorphism \begin{gather*} 
(T^*\R^2,d\lambda_{can}) = (\R^2_{\mathbf{q}} \times \R^2_{\mathbf{p}},d(\mathbf{p}\,d\mathbf{q})) \rightarrow (\C^2,\omega_0),\\
 (q_1,q_2,p_1,p_2) \mapsto (x_1,y_1,x_2,y_2)=(p_1,q_1,p_2,q_2).
\end{gather*}
 In particular the section $\Gamma_{q_1\,dq_2} \subset T^*\R^2$ is identified with the section $\{x_1=0,x_2=y_1\} \subset \C^2$ over $\mathfrak{Im}\C^2$. 
Second, the symplectomorphism
\begin{gather*}
\C^2 \to \C^2,\\
(x_1,y_1,x_2,y_2) \mapsto (x_1,y_1+y_2,y_2,x_1-x_2)
\end{gather*}
sends the symplectic plane $\{0\} \times \C$ to a linear plane that corresponds to the latter section $\Gamma_{q_1\,dq_2}$ over the canonical projection $\C^2 \to \mathfrak{Im}\C^2$.

Case (2): First, there is a canonical symplectomorphism
\begin{gather*}
(T^*(\R \times S^1),d\lambda_{can})=(\R_q \times S^1_\theta \times \R^2_{p_q,p},d(\mathbf{p}\,d\mathbf{q})) \rightarrow (T^*S^1 \times \C,d(p\,d\theta)\oplus \omega_0)\\
 (q,\theta,p_q,p) \mapsto ((\theta,p),(x,y))= ((\theta,p),(p_q,q)).
\end{gather*} Second, we consider the symplectomorphism
\begin{gather*}
(T^*S^1 \times \C, d(p\,d\theta)\oplus\omega_0) \to (T^*S^1 \times \C, d(p\,d\theta)\oplus\omega_0) \\
((\theta,p), (x,y) ) \mapsto \left(((\theta+x+y),p),(x-p)+i\left(y+p\right) \right)
\end{gather*}
which sends $T^*S^1 \times \{0\}$ to the symplectic section
$\{x=-y,p=y\}$
 
over the canonical projection $T^*S^1 \times \C \to S^1 \times \mathfrak{ Im}\C$.

As in Case (1), we may consider the case when the primitive of the area form on $\tilde{\Sigma}=S^1_\theta \times \R_q$ is equal to $\beta=q\,d\theta-q\,dq$, which corresponds to the above section $\{x=-y,p=y\}$. We again get the sought symplectomorphism.
\end{proof}

\begin{proof}[Finishing the proof of Part (2) of Theorem \ref{thm:a}]
 We start by invoking Lemma \ref{lma:klein}, which implies that there is an infinite cover $(\tilde{V},\lambda) \to (V,\lambda)$ under which $L$ lifts, where the cover is induced by a cover $\tilde{\Sigma}_g \to \tilde{\Sigma}$ with fundamental group that is cyclic and either trivial or of infinite rank. The above discussion implies that $\tilde{V}$ is the complement of an open tubular neighbourhood of a symplectic section in $T^*\tilde{\Sigma}$ where $\tilde{\Sigma}=\R^2$ or $S^1 \times \R$.
Further, by Lemma \ref{lma:lift}, there is a symplectic embedding of $(\tilde{V},d\lambda)$ into $(\C^* \times \C,\omega_0)$ in the first case, and into $(T^*S^1 \times \C^*,d(p\,d\theta)\oplus \omega_0)$ in the second case. Since $\tilde{V}$ is realised as the complement of a tubular neighbourhood of the symplectic section $\{0\} \times \C \subset \C \times \C$ and $T^*S^1 \times \{0\} \subset T^*S^1 \times \C$, respectively, the embedding of $\tilde{V}$ is clearly a homotopy equivalence. 

The existence of a Klein bottle in a McDuff domain is excluded by Proposition \ref{prp:klein} in combination with the result in the above paragraph. We are thus left with the case when $L \subset V$ is a weakly exact torus. 

If $L$ is a weakly exact Lagrangian torus, then there is no lift to $(\C^* \times \C,\omega_0)$ by Theorem \ref{thm:lalondesikorav}. Indeed, it is easy to see that the latter symplectic manifold admits a symplectic embedding into $T^*(S^1 \times \R)$, where the embedding moreover is a homotopy equivalence. We conclude that $L$ is a weakly exact torus that admits a weakly exact lift to a cover $(\tilde{V}=T^*S^1 \times \C^*,d(p\,d\theta) \oplus \omega_0)$ of $V$. Moreover, this cover is induced by a cover $\tilde{\Sigma}_g \to \Sigma_g$ of the base of the bundle $V \to \Sigma_g$ that corresponds to an infinite cyclic subgroup $\Z \cong H \subset \pi_1(\Sigma_g)$ equal to the image of $\pi_1(L)$ under the map of fundamental groups induced by the composition $L \to \Sigma_g$. 

 We can consider a standard circle-bundle torus in the non-compact McDuff domain $\tilde{V}\to\tilde{\Sigma}$. By the classification result Theorem \ref{thm:HamiltonianIsotopy} any two such tori are smoothly isotopic through weakly exact Lagrangian tori. If $L$ is not merely weakly exact, but even exact, then we proceed as follows. By \cite[Section 4.1]{Cieliebak:Anasov} there is an exact circle-bundle torus in $\tilde{V}$ that is constructed from the unique closed embedded geodesic in $\tilde{\Sigma}_g \cong \R \times S^1$. After making the initial neighbourhood $\mathcal{O}_{0_\Sigma}$ sufficiently small, we can assume that the Hamiltonian isotopy provided by Theorem \ref{thm:HamiltonianIsotopy} is confined to $\tilde{V} \cong T^*S^1 \times \C^*$. 

What remains is to show that the infinite cover $\tilde{V}\to V$ factorises through a $k$-fold cover $\tilde{V} \to \tilde{V}_k \to V$, where $k \gg 0$ is sufficiently large, under which the above isotopies of tori $\tilde{L}_t \subset \tilde{V}$ remains embedded when composed with the covering map $\tilde{V} \to \tilde{V}_k$. To that end, we will need the following corollary of a result \cite[Theorem 3.3]{Scott} by Scott.
\begin{cor}Let $p \colon \tilde{\Sigma}_g \to \Sigma_g$ be a possibly infinite cover of a compact surface that corresponds to a finitely generated subgroup $H \subset \pi_1(\Sigma_g)$. For any finite set of points $\{x_1,\ldots,x_N\} \subset \Sigma_g$ and finite subset of the fibres $S \subset \bigcup_{i} p^{-1}(x_i)$, the cover $p$ can be factorised through a \emph{finite} $k$-fold cover $\tilde{\Sigma}_g \xrightarrow{\tilde{p}} (\Sigma_g)_k \to \Sigma_g$ under which $\tilde{p}|_S$ is injective.
\end{cor}
\begin{proof}
Any larger subgroup $H \subset G \subset \pi_1(\Sigma_g)$ yields a cover $p_G \colon (\Sigma_g)_G \to \Sigma_g$ through which the original cover $p$ factorises as a composition
$$\tilde{\Sigma}_g \xrightarrow{\tilde{p}_G} (\Sigma_g)_G \xrightarrow{p_G} \Sigma_g.$$
The goal is to find $G$ which is sufficiently large in order for it to be of finite index (i.e.~the corresponding cover is finite), but which is small enough so that the image of the finite subset $ S$ remains embedded (i.e.~the cover should not collapse too many fibres of $\tilde{\Sigma}_g \to \Sigma_g$).

The fibre of $p^{-1}(x_i)$ can be identified with the cosets $\pi_1(\Sigma_g)/H$, if the base-point is chosen at $x_i$. Two elements $aH \neq bH \in \pi_1(\Sigma_g)/H$ of $S$ are identified in $\pi_1(\Sigma_g)/G$ if and only if $a=bg$ for some $g \in G$. Using \cite[Theorem 3.3]{Scott} we can find an extension $G \supset H$ where $G \subset \pi_1(\Sigma_g)$ is of finite index, but does not contain the element $b^{-1}a \in \pi_1(\Sigma_g) \setminus H$. For such an extension, it thus follows that the fibres $aH,bH \in S$ are not identified under $\tilde{p}_G$.

We can then iterate the argument with $\Sigma_g$ replaced by this finite cover $(\Sigma_g)_G$ in order to find a larger cover under which more elements in $S$ become separated. To that end, we just need the basic fact that the intersection of a finite set of finite-index subgroups of $\pi_1(\Sigma_g)$ that contain $H$ is again a subgroup of the same type. Since the set $S$ is finite, this process terminates.
\end{proof}
Using a compactness argument, we can then readily pass to a sufficiently large, but finite, cover $\tilde{V} \xrightarrow{\tilde{p}} \tilde{V}_k \to V$ for which the images $\tilde{p}(\tilde{L}_t)$ remain embedded for all $t$.
\end{proof}

\section{Conditions for non-vanishing wrapped Floer cohomology (Theorems \ref{thm:b} and \ref{thm:b2})}
As in the assumptions of the theorem, we let $(X,\lambda)$ be a connected Liouville domain with boundary components $\partial X=\bigsqcup_{i \in \pi_0(\partial X)} (\partial X)_i$ and $L \subset X$ a connected exact Lagrangian with cylindrical ends, whose Legendrian boundary has an induced decomposition $\partial L=\bigsqcup_{i \in I_L} (\partial L)_i$. Here $I_L\subset \pi_0(\partial X)$ are the connected components of $\partial X$ that have non-empty intersection with $\partial L$, i.e.~ $(\partial L)_i=\partial L \cap (\partial X)_i \neq \emptyset$. We extend $L \subset X$ to a properly embedded exact Lagrangian of the completion $\hat{L} \subset (\hat{X},\lambda)$ by adjoining the Lagrangian cylinders
$$[0,+\infty) \times \Lambda \subset ([0,+\infty) \times Y,e^\tau\alpha)$$
contained inside the cylindrical end.

\subsection{ Conventions for the Floer homological set-up}

Here we describe the setup of the Floer complexes that we will be using for defining symplectic cohomology and wrapped Floer cohomology. We mostly follow the conventions of Ritter from \cite{Ritter:TQFT}, except that our convention of the action of the generators have the opposite sign compared to his.

\subsubsection{Wrapping and direct limit}

Symplectic cohomology and wrapped Floer cohomology are both defined as certain direct limits of Floer complexes that are induced by time-dependent Hamiltonians $H^a_t$ on the completion $(\hat{X},\lambda)$ that satisfy appropriate growth conditions. Our convention is that $H^a_t \colon \hat{X} \to \R$ takes the form $H^a_t=ae^{\tau}+b$ in the subset
$$\{ \tau \ge 0\} \subset ((-\infty,+\infty)_\tau \times Y,e^\tau \alpha) \subset (\hat{X},\lambda)$$
of the collar and cylindrical end, and where the sequence
$$a \in \{a_0 < a_1< a_2 <\cdots \}$$
is generic and satisfies $\lim_{i \to +\infty} a_i = +\infty$. We will impose some additional requirements on these systems of Hamiltonians near the compact part $X$, which are described in Subsection \ref{sec:hamiltonian} below.

 In the case of symplectic cohomology $SH^*(X)$, the complexes are generated by periodic Hamiltonian time-1 orbits, while in the case of wrapped Floer cohomology $HW^*(L,L)$ of a Lagrangian $L$, the complexes are generated by Hamiltonian time-1 chords that start and end on $L$. More precisely, for each slope $a$, there are complexes $SC^*(X;a)$ and $CW^*(L,L;a)$ for which the generators are induced by the Hamiltonian $H^a_t$. The symplectic and wrapped Floer cohomologies are then defined as certain direct limits as $a \to +\infty$. One should note that there are several different constructions that give rise to the same homology groups in the end.

A crucial feature of this set-up is that there are natural continuation maps from the complex of slope $a_i$ to the complex of slope $a_j$ whenever $a_j>a_i$. The original version of symplectic cohomology was defined as a direct limit $\lim_{a \to +\infty} SH^*(X;a)$ of the homology groups induced by the continuation maps; the wrapped Floer cohomology has a completely analogous construction. This approach was taken by e.g.~Ritter in \cite{Ritter:TQFT}. A more modern approach is to define a single complex $SC^*(X)=\lim_{a \to +\infty} SC^*(X;a)$ given as a homotopy co-limit of the complexes defined for finite slopes; see e.g.~work by Abouzaid--Seidel \cite{Abouzaid:OpenString}

The end results of these different constructions are the same, in the sense that they give isomorphic homologies. In addition, in either construction, the concerned complexes all consist of generators that are Hamiltonian orbits or chords for the Hamiltonians $H^a_t$. In the following we will consider Floer complexes $SC^*(X;a)$ and $CW^*(L,L;a)$ for some fixed slope $a$, which will be suppressed from the notation. The analysis performed for these complexes can then be used to derive the results for either of the two versions of the direct limit. 

\subsubsection{Action of generators and Floer strips}

The Floer cylinders and strips counted by the differentials in our complexes are described in Section \ref{sec:floerstrips} of Appendix \ref{sec:appendix}. More precisely, the Floer strip equation is given in Equation \eqref{eq:floerstrip}, while the action conventions that we use for the periodic Hamiltonian orbits and Hamiltonian chords with endpoints on $L$ are given in Equations \eqref{eq:actionclosed} and \eqref{eq:actionopen}, respectively. This means that the differential \emph{decreases} the action. Note that our action differs from that in \cite{Ritter:TQFT} by a minus sign. Here we summarise the most important features of our conventions:
\begin{itemize}
\item When $H$ is an autonomous $C^2$-small Hamiltonian, the 1-periodic orbits can be assumed to be constant $x \in \OP{Crit}(H)$. The differential corresponds to a deformation of the Morse \emph{co}homology differential for $H$ which counts flow-lines of the gradient $\nabla H$. However, unlike the usual convention in Morse cohomology, the action of $x$ is equal to $-H(x)$ with our definition.
\item When $H$ is an autonomous $C^2$-small Hamiltonian for which $H|_L$ is Morse, the time-1 Hamiltonian chords from $L$ to $L$ can be assumed to be constant $x \in \OP{Crit}(H|_L)$ with action equal to $-H(x)$. The differential corresponds to a deformation of the Morse \emph{co}homology differential for $H|_L$.
\item If $H$ in $([a,b]_\tau \times Y,d(e^\tau\alpha))$ is $C^0$-small, only depends on $\tau$, and satisfies $\partial_\tau H>0$ and $\partial_\tau^2 H > 0$, then the periodic Hamiltonian orbits, and Hamiltonian chords with endpoints on the Lagrangian $[a,b] \times \Lambda$, correspond to periodic Reeb orbits in $(Y,\alpha)=(\partial X,\lambda|_{T\partial X})$ and Reeb chords on $\partial L \subset (Y,\alpha)$, respectively. The actions of these generators are roughly equal to the Reeb length in the contact manifold (i.e.~integration of $\alpha$ along the Reeb chord) if the Lagrangian is endowed with the vanishing potential (i.e.~primitive of $e^\tau\alpha|_{TL}$).
\item Even if the grading will not play any important role in our proofs, it can still be enlightening to pin-point the most natural grading convention. We use the same convention as in \cite{Ritter:TQFT}, where the aforementioned generators that correspond to critical points of $H$ are graded as in Morse cohomology. That is, the degree of $x \in \OP{Crit}(H)$ is equal to the Morse index of $x$, and the differential is of degree $+1$. (In general, however, grading must be taken in $\Z_2$.)
\end{itemize}

\subsubsection{Behaviour of the Hamiltonian near the boundary, and decomposition into small and large action complexes}
\label{sec:hamiltonian}
 
Recall the requirement that $H^a_t=a\tau+b$ in the subset $\{\tau \ge 0\} = \hat{X} \setminus \OP{int}X$. In order to gain additional control of the Floer complex we need to impose constraints on the Hamiltonians $H^a_t$ in $X$ as well. In particular, we need some precise control on 
$H^a_t$ in the collar 
$$((- 4 \epsilon,0]_\tau \times Y,e^\tau \alpha) \subset (X,\lambda)$$
in the complement of the cylindrical end, where $\epsilon>0$ is sufficiently small. In particular, we may assume that the Lagrangian $L \subset (X,\lambda)$ is cylindrical in this collar, i.e.~that it coincides with $(-4\epsilon,0] \times \Lambda$ there for a Legendrian $\Lambda \subset (Y,\alpha)$. For any $S \le T \le 0$ we write
$$X_{T}\coloneqq \phi^{-T}_\zeta(X),\:\:X_{[S,T]}=X_T \setminus \OP{int}{X_S}, \:\: \text{and} \:\: X_{(S,T)}=(\OP{int} X_T) \setminus X_S.$$
(Note that $\phi^{-t}$ is just translation of the $\tau$-coordinate by $-t$ in the above symplectisation coordinates.) Our additional requirements are that: 
\begin{itemize}
\item $H^a_t$ is a $C^{0}$-small function on $X$, which moreover is
\begin{itemize}
 \item autonomous (independent of $t$), independent of the parameter $a>0$, and Morse in all of $X_{-\epsilon/2}$; 
 \item non-negative everywhere except in $X_{[-\epsilon-\epsilon/3,-\epsilon+\epsilon/3]}$
 \end{itemize}
\item The partial derivative in the symplectisation direction $\partial_\tau$ satisfies:
\begin{itemize}
\item In $\hat{X} \setminus X_{-\epsilon}$ we have $\partial_\tau H^a_t > 0$ and $\partial_\tau^2 H^a_t>0$; 
\item In the hypersurface $\{-2\epsilon, -\epsilon\} \times Y$ we have $\partial_\tau H^a_t = 0$;
\item In $X_{(-2\epsilon,-\epsilon)}$ we have $\partial_\tau H^a_t < 0$;
\item In $X_{(-4\epsilon,-2\epsilon)}$ we have $\partial_\tau H^a_t > 0$.
\end{itemize}
\item $H^a_t=h(\tau)+\delta f$ in $X_{[-2\epsilon-\epsilon/3, -2\epsilon/3]}$ for some Morse function $f \colon Y \to (-1,0]$, where the parameter $\delta>0$ is sufficiently small;
\item $H^a_t=a' e^\tau+b'$ holds near the hypersurface $\{-2\epsilon-\epsilon/2,-\epsilon/2\} \times Y$, where $a'>0$ and $b'\ge0 $ are both sufficiently small; and
\item $0<\min_{X_{-3\epsilon}}H^a_t < \max_{X_{-3\epsilon}}H^a_t = h(-3\epsilon)$.
 
\end{itemize}
 
This type of Hamiltonian gives rise to the following generators.

\emph{Generators of symplectic cohomology:} For a generic sequence $a_1<a_2<\ldots$ the generators of periodic Hamiltonian orbits in $\hat{X}$ of $H^{a_i}_t$ all live inside $X$, and are of the following four types:
\begin{enumerate}[label=(\Alph*)]
 \item \emph{Generators in $X_{-3\epsilon}$:} These generators all have actions contained in the interval $(-h(-3\epsilon),0) \subset \R$;
 \item \emph{Generators in $\{-2\epsilon\} \times Y$:} These generators correspond to the critical points of the Morse function $f$ on $Y$, where the action of a generator $x \in \OP{Crit}(f)$ is equal to $\mathcal{A}(x)= -h(-2\epsilon) - \delta f(x)$, and thus
 $$-h(-2\epsilon) \le \mathcal{A}(x) \le -h(-3\epsilon) < 0,$$
 for $\delta>0$ sufficiently small. Note that the the degrees of these generators are one larger than the Morse index of the corresponding critical point in $f$;
 \item \emph{Generators in $\{-\epsilon\} \times Y$:} These generators also correspond to critical points of the Morse function $f$ on $Y$, where the action of $x \in \OP{Crit}(f)$ is equal to $\mathcal{A}(x)=-h(-\epsilon) - \delta f(x) $, and thus
 $$0<-h(-\epsilon) \le \mathcal{A}(x) \le -h(-\epsilon)+\delta.$$
 Here the degree of the generator is equal to the Morse index of the corresponding critical point.
 \item \emph{Generators in $\hat{X}\setminus X_{-\epsilon/2}$:} These generators correspond to the periodic Reeb orbits of $(Y,\alpha)$ of Reeb length less than $a_i$, and have action roughly equal to the corresponding Reeb lengths. Since $H^{a_i}_t$ is assumed to be $C^0$-small in $X\setminus X_{-\epsilon/2}$, we can assume that these generators all have action strictly greater than all generators described in (A)--(C) above.
\end{enumerate}

We then proceed with the analogous action computations in the case of wrapped Floer homology $HW^*(L,L)$. In this case, however, the action also depends on a choice of primitive $g$ of the pull-back of the Liouville form $\lambda|_{TL}=dg$ to $L$. The following basic lemma shows that we can replace our Lagrangian with one for which the primitive can be taken to be arbitrarily small, and moreover vanishing near the boundary.
\begin{lma}
 After a Hamiltonian isotopy of $L\subset (X,\lambda)$ that is supported away from the boundary, we may assume that the pull-back $\lambda|_{TL} \in \Omega^1(L)$ has a primitive that vanishes on the entire boundary $\partial L \subset Y=\partial X$, while it is arbitrarily $C^0$-small in the interior.
\end{lma}
\begin{proof}
The negative Liouville flow applied to $L$ gives rise to an isotopy $\phi^{-t}_\zeta(L)$ through exact Lagrangians that, since $L$ is tangent to the Liouville flow near $\partial X$, can be extended to an exact Lagrangian isotopy that is fixed near the boundary. Moreover, since the isotopy preserves exactness, a standard result implies that it is generated by an ambient Hamiltonian isotopy.

Further, since $(\phi^{-t}_\zeta)^*\lambda=e^{-t}\lambda$, it follows that we can find a primitive that is arbitrarily $C^0$-small after such an isotopy. However, even if any primitive necessarily is locally constant in each cylindrical region near the boundary, there might still be action-differences between different boundary components. This can finally be amended by an appropriate slight wrapping of each separate component of the collar of $L$. More precisely, near each component of the collar
$$ (-4\epsilon,0]\times\Lambda \subset ((-4\epsilon,0]_\tau \times Y,d(e^\tau \alpha)),$$
we apply the Hamiltonian isotopy induced by a Hamiltonian vector-field of the form $\rho_i(\tau)R_\alpha$. Here $R_\alpha \in \Gamma(TY)$ is the Reeb vector field associated to $\alpha$, and $\rho_i(\tau)$ is a suitable function with compact support in $(-4\epsilon,0) \subset \R_\tau$ depending on the component $i \in \pi_0(\partial L)$. The effect of this wrapping on the action, i.e.~the primitive of the pull-back of $\lambda$ to the Lagrangian, is a computation that we leave to the reader.
\end{proof}

Since wrapped Floer homology is invariant under Hamiltonian isotopies of the Lagrangian, we will in the following assume that $L$ satisfies the properties of the above lemma.

\emph{Generators of wrapped Floer homology:} For a Lagrangian $L$ which is cylindrical in the subset $X_{[-4\epsilon,0]}$ we have generators that are Hamiltonian chords from $L$ to $L$ of the three different types as above, but where the Morse functions instead are the corresponding restrictions to $L$, and where the generators in the last bullet-point correspond to Reeb chords from $\Lambda \subset (Y,\alpha)$ to itself of length less than $a_i$ instead of periodic Reeb orbits.

For the systems of Hamiltonians as above there are well-known induced cone structures
\begin{gather*}
SC^*(X)=\OP{Cone}(\delta),\\
 \delta \colon SC^*_\infty(X) \to SC^*_0(X),
 \end{gather*}
of the symplectic cohomology complex as well as
\begin{gather*}
CW^*(L,L)=\OP{Cone}(\delta),\\
 \delta \colon CW^*_\infty(L,L) \to CW^*_0(L,L)
 \end{gather*}
of the wrapped Floer cohomology complex. Here the source complexes $SC^*_\infty(X)$ and $CW^*_\infty(L,L)$ consists of the generators contained in the subset $ \hat{X} \setminus X_{-\epsilon/2}$ of the cylindrical end (i.e.~corresponding to Reeb orbits and chords), while the target complexes $SC^*_0(X)$ and $CW^*_0(L,L)$ consist of generators contained in $X_{ -\epsilon/2}$ (i.e.~low energy Hamiltonian orbits and chords).
\begin{lma}[Lemma 2.1 in \cite{CieliebakOancea}]
\label{lma:zeropart}
The sub-spaces $SC^*_0(X)\subset SC^*(X)$ and $CW^*_0(L,L) \subset CW^*(L,L)$ are sub-complexes whose homology computes $ H^*(X)$ and $ H^*(L)$, respectively.
\end{lma}
\begin{proof}
\emph{The sub-complex property:} This follows from the action computations in Subsection \ref{sec:hamiltonian} above, together with the fact that the Floer differential decreases the action. 

\emph{The homology computation:} This is a standard application of Floer's computation from \cite{Floer:Morse}. To that end, we need the no-escape lemma for Floer trajectories; see \cite[Lemma 2.2]{CieliebakOancea} for the case of symplectic cohomology, and \cite[Lemma 7.2]{Abouzaid:OpenString} or \cite[Lemma 3.1]{EkholmOancea} for the case of wrapped Floer cohomology. By the no-escape lemma, a Floer strip that connects two generators contained inside $ X_{-\epsilon/2}$ must be contained entirely inside the same subset. Finally, since the Hamiltonian is $C^2$-small, the complex can be assumed to coincide with the Morse cohomology complex by Floer's original argument.
\end{proof}
\subsection{ Decomposition induced by components of the contact boundary.}
 Recall that we denote the components of the contact boundary by $\partial X=\bigsqcup_{i \in \pi_0(\partial X)} (\partial X)_i,$ which induces a decomposition of the boundary of any Lagrangian $(\partial L)_i\coloneqq \partial L \cap (\partial X)_i$. This decomposition induces a natural decomposition of the vector spaces
$$ SC^*_\infty(X)=\bigoplus_{i \in \pi_0(\partial X)} SC^*_\infty((\partial X)_i)$$
of symplectic cohomology complex and
$$ CW^*_\infty(L,L)=\bigoplus_{i \in I_L} CW^*_\infty((\partial L)_i,(\partial L)_i)$$
of the wrapped Floer cohomology complex. The following basic neck-stretching argument implies that these decompositions also hold on the level of complexes.
\begin{lma}
\label{lma:complexdecomp}
The complexes $SC^*_\infty(X)$ and $CW^*_\infty(L,L)$ both respect the decompositions
\begin{gather*}
SC^*_\infty(X)=\bigoplus_{i \in \pi_0(\partial X)} SC^*_\infty((\partial X)_i)\\
\text{and}\\
CW^*_\infty(L,L)=\bigoplus_{i \in I_L} CW^*_\infty((\partial L)_i,(\partial L)_i)
\end{gather*}
corresponding to the decompositions of boundary components.
\end{lma}
\begin{proof}
Since the generators in $SC^*_\infty(X)$ and $CW^*_\infty(L,L)$ all correspond to Reeb orbit/chord generators, they can be assumed to all have action bounded from below by some fixed positive number. The statement is now a direct consequence of a standard neck-stretching argument which shows that the differential of $SC^*_\infty((\partial X)_i)$ and $CW^*_\infty((\partial L)_i,(\partial L)_i)$ cannot output a generator in some other component if the almost complex structure is taken to be cylindrical near the the hypersurface
$$\Sigma = \{\tau= -\epsilon/2\} = \partial X_{-\epsilon/2
} \subset \hat{X}.$$
 To that end, we apply Lemma \ref{lma:actionstretch} to the decomposition of $(\hat{X},\lambda)$ induced by $\Sigma_i \subset \Sigma$, where $\Sigma_i$ contains all connected components \emph{except} the one corresponding to $i\in \pi_0(\partial X)$. 
\end{proof}
 
For any $i \in \pi_0(\partial X)$ and $i \in I_L$, respectively, we have the sub-spaces
$$ SC^*_{0,i}(X) \subset SC^*_{0}(X)$$
of the symplectic cohomology complex and
$$CW^*_{0,i}(L,L) \subset CW^*_0(L,L)$$
of the wrapped Floer cohomology complex, each consisting of the generators that are contained inside
$$X_{-\epsilon-\epsilon/2} \cup U_i.$$
where
$$ U_i \subset X_{[-\epsilon-\epsilon/2,0]}=[-\epsilon-\epsilon/2,0] \times Y$$
is the connected component of the collar of $\partial X$ that corresponds to $i \in \pi_0(\partial X)$.
\begin{lma}
 The inclusions of vector subspaces
 $$ SC^*_{0,i}(X) \subset SC^*_0(X) \:\:\:\: \text{and} \:\:\:\:
 CW^*_{0,i}(L,L) \subset CW^*_0(L,L)$$ are sub-complexes that on homology give rise to the canonical maps
$$ H^*(X,\partial X \setminus (\partial X)_i) \to H^*(X) \:\:\:\: \text{and} \:\:\:\:
H^*(L,\partial L \setminus (\partial L)_i
) \to H^*(L)$$
in singular cohomology of these spaces.
\end{lma}
\begin{proof}

\emph{The subcomplex property:} The generators contained inside $X_{-2\epsilon}$ form a subcomplex simply by the action computations from Subsection \ref{sec:hamiltonian}. When adjoining the generators in $U_i \subset X$ we again get a sub-complex, as follows by a neck-stretching argument using Lemma \ref{lma:actionstretch} similar to that in the proof of Lemma \ref{lma:complexdecomp}, but applied to the hypersurface $\Sigma_i$ consisting of all connected components of
$$\Sigma=\{\tau=-2\epsilon-\epsilon/2\}=\partial X_{-2\epsilon-\epsilon/2}$$
\emph{except} the component corresponding to $i \in \pi_0(\partial X).$ To that end, note that all generators in $\{\tau=-\epsilon\}=\partial X_{-\epsilon}$ are of positive action. 

\emph{The computations of the homology and inclusion maps:} This follows from Floer's classical computation \cite{Floer:Morse}, by which the Floer homology complex computes the Morse co homology for sufficiently $C^1$-small Hamiltonians and suitable almost complex structures. Here we again need the no-escape lemma for Floer strips as in the proof of Lemma \ref{lma:zeropart}
\end{proof}
Theorems \ref{thm:b} and \ref{thm:b2} are now direct consequences of the following result.
\begin{prp}
The restrictions
$$\delta_i \colon SC^*_\infty((\partial X)_i) \to SC^*_0(X) \:\:\:\: \text{and} \:\:\:\: \delta_i \colon CW^*_\infty((\partial L)_i,(\partial L)_i) \to CW^*_0(L,L)$$
of the map $\delta$ takes values in $SC^*_{0,i}(X)$ and $CW^*_{0,i}(L,L)$, respectively.
\end{prp}
\begin{proof}
We need to show that the differential of a generator in $SC^*_\infty((\partial X)_i) \subset SC^*(X)$ and $CW^*_\infty((\partial L)_i,(\partial L)_i) \subset CW^*(L,L)$ cannot have a non-vanishing coefficient for a generator in a component of $\{\tau= -\epsilon\}= \partial X_{-\epsilon}$ that does not correspond to $i \in \pi_0(\partial X)$. 

To that end, we again do a neck-stretching argument. Namely, if the almost complex structure is chosen to be cylindrical near the contact-type hypersurface
$$\Sigma = \{\tau = -2\epsilon-\epsilon/2 \} = \partial X_{-2\epsilon-\epsilon/2 } \subset X,$$
then we can then invoke the neck-stretching Lemma \ref{lma:actionstretch} to the union $\Sigma_i \subset \Sigma$ all connected components \emph{except} the one corresponding to $i \in \pi_0(\partial X)$ to obtain the result. Here one should note that all generators contained in the hypersurface $\{\tau=-\epsilon\}=\partial X_{-\epsilon}$ are of positive action. 
\end{proof}

\section{Conditions for vanishing wrapped Floer cohomology (Theorem \ref{thm:c})}

The vanishing of the symplectic homology of a Weinstein manifold whose completion is a product $(\hat{W} \times \C,\lambda_W \oplus \lambda_0)$ with a trivial $(\C,\lambda_0)$-factor with $\lambda_0=\frac{1}{2}(x\,dy-y\,dx)$ was first shown by Cieliebak in \cite{Cieliebak:Handle}. These are the completions of the so-called \textbf{subcritical Weinstein domains}, which is a class that contains e.g.~the standard symplectic Darboux ball. It then follows by Ritter's result \cite{Ritter:TQFT} that the wrapped Floer complexes also are acyclic in these manifolds, since the wrapped Floer cohomology is a module over the unital symplectic cohomology ring.

We need the following generalisation of the above vanishing criterion, formulated in terms of the existence of a contractible positive loop of the Legendrian boundary. Recall that a smooth loop of embedded Legendrians is called a \textbf{positive loop} if the normal vector field is always positively transverse to the contact distribution; a loop of Legendrians is said to be \textbf{contractible} if there is a two-parameter family of loops of Legendrians that connects the positive loop with a constant loop. Note that we do not ask for a positive contractible loop to be contractible \emph{through positive loops} (in fact, that is never possible). 
\begin{thm}
 \label{thm:vanishing}
 Let $L \subset (X^{2n},\lambda)$ be a connected Lagrangian with cylindrical ends. Assume that a non-empty union of Legendrian boundary components $\Lambda \subset \partial L \subset Y=\partial X$ admits a positive contractible loop in the complement of the remaining components $\partial L \setminus \Lambda$. Then $HW^*(L,L)=0$.
\end{thm}
\begin{proof}
If the entire Legendrian boundary of a Lagrangian $L \subset (X,\lambda)$ admits a positive contractible contact loop, then $HW^*(L,L)$ vanishes by \cite[Theorem 1.15]{Chantraine:Positive} or, alternatively, the stronger result \cite[Theorem 1.2]{HedickeCantKilgore} by Cant--Hedicke--Kilgore.

In the case when merely some of the components live in a positive loop, while others are fixed, we perform the following construction to the reduce the situation to the previous case. Consider the disjoint union $(X', \lambda')=(X,\lambda) \sqcup (X,\lambda)$ consisting of two disjoint copies of $(X,\lambda)$ together with the induced disjoint union $L'=L \sqcup L \subset (X', \lambda')$ of two copies of $L$. We can attach a generalised Weinstein handle $T^*((\partial L \setminus \Lambda) \times [-1,1])$ to $(X',\lambda)$ along the components that are fixed by the positive loop to create a Liouville manifold $(X'', \lambda'')$. We refer to e.g.~work by Husin \cite{Husin}, as well as ~\cite{DimitroglouRizell2022}, for further description of the construction of this generalised Weinstein handle attachment. 

The new Liouville manifold $(X'', \lambda'')$ admits an exact Lagrangian
$$L'' = L' \cup 0_{(\partial L \setminus \Lambda) \times [-1,1]}$$
with cylindrical ends, obtained by ``capping off'' $L'$ by adding the skeleton (i.e.~zero-section) of the generalised handle. Note that $L''$ now satisfies the property that all its entire Legendrian boundary $\Lambda \sqcup \Lambda$ sits inside a positive contractible loop. Hence, $HW^*(L'',L'')=0$, as we proved in the first case above. 

It now follows from an application of Viterbo functoriality that $HW^*(L',L')=HW^*(L,L) \oplus HW^*(L,L)$ vanishes as well; indeed, $L'' \setminus L'$ is an exact Lagrangian cobordism in the Liouville cobordism $X'' \setminus X'$, which by Viterbo functoriality induces a unital ring morphism $HW^*(L'',L'') \to HW^*(L',L')$. We refer to work by Abouzaid--Seidel \cite{Abouzaid:OpenString} for Viterbo functoriality in the case of wrapped Floer homology.
\end{proof}

By a \textbf{ closed contact Darboux ball} we mean a closed ball with smooth boundary contained inside the standard contact vector space
$$\left(\R^n_{\mathbf{x}} \times \R^n_{\mathbf{y}} \times \R_z,dz-\sum_i y_i\,dx_i\right),$$
whose boundary is transverse to the contact vector field
$$ V_{rad} \coloneqq z\partial_z+\frac{1}{2}\sum_i(x_i\partial_{x_i}+y_i\partial_{y_i})$$
 that corresponds to the rescaling $$(\mathbf{x},\mathbf{y},z) \mapsto \rho_t(\mathbf{x},\mathbf{y},z)=(e^{t/2}\mathbf{x},e^{t/2}\mathbf{y},e^tz).$$
In particular, the boundary of such a ball is \textbf{convex} in the sense of Giroux \cite{Giroux}; it can be seen that the interiors of these Darboux balls all are contactomorphic to entire ambient standard contact vector space. Combining the above with a recent results by Hedicke--Shelukhin \cite{HedickeShelukhin} we obtain:
\begin{cor}
\label{cor:vanishing}
Let $L \subset (X^{2n},\lambda)$ be a connected Lagrangian with cylindrical ends. Assume that a non-empty union of Legendrian boundary components $\Lambda \subset \partial L \subset Y=\partial X$ is contained inside a contact sub-domain $(Y_{sc} \setminus D ) \subset Y$ where $Y_{sc} \setminus D$ is contactomorphic to the ideal contact boundary $Y_{sc}=\partial_\infty(\hat{W}^{2(n-1)} \times \C)$ of a subcritical Weinstein domain with a finite number of disjoint closed contact Darboux balls $ D \subset Y_{sc}$ removed. If either $n \ge 3$, or $n=2$ and $Y_{sc}=S^3$, then $HW^*(L,L)=0$.
\end{cor}
\begin{rmk}
It is expected that the additional hypothesis in the case $n=2$ is not needed, i.e.~ that it suffices that $Y_{sc}$ is subcritical in all dimensions $n \ge 2$; see \cite[Remark 1.6(i)]{HedickeShelukhin}.
\end{rmk}
\begin{proof}
By assumption $\Lambda^{n-1} \subset Y_{sc}^{2n-1} \setminus D$ where $Y_{sc}^{2n-1}=\partial_\infty(\hat{W}^{2(n-1)} \times \C)$ is the ideal contact boundary of a subcritical Weinstein domain, and $D$ is a finite disjoint union of finite Darboux balls.

 First we show that $\Lambda$ admits a positive contractible loop confined to $Y_{sc}$. Then we show that the same is true also in $Y_{sc} \setminus D$. The result then follows from Theorem \ref{thm:vanishing}. 

It was shown in \cite[Theorem 1.5]{HedickeShelukhin} that the ideal contact boundary $Y_{sc}^{2n-1}$ of a subcritical Weinstein manifold is non-orderable whenever $n\ge 3$, i.e.~such a contact manifold admit a positive loop of global contactomorphisms that can be contracted through contactomorphisms. It automatically follows that any Legendrian inside $Y_{sc}$ also admits a positive contractible loop.

When $n=2$ the same is known to be true when $Y^{3}=S^3$ is the standard contact 3-sphere; see work by Eliashberg--Kim--Polterovich \cite{GeometryOrderability}. 

 We now argue that the positive loop and its contraction both can be assumed to live inside the complement $Y_{sc}\setminus D$ of the Darboux balls. When $n \ge 3$, this follows from the well-known fact that $Y_{sc}\setminus D$ is contactomorphic to $Y_{sc} \setminus P$ where $P$ is a finite set consisting of $|P|=\dim H_0(D)$ number of points. Indeed, since $n \ge 3$, a generic position argument implies that both the positive loop of $\Lambda$ and its contraction can be assumed to miss the finite set of points $P$. To see that removing a closed ball is the same as removing a point, one can use a cut-off of the vector field $V_{rad}$ described above, i.e.~the convexity of the boundary $\partial D$, to produce the sought contactomorphism.

 The aforementioned general position argument fails when $n=2$. However, since we have made the additional assumption that $Y_{sc}=S^3$, we can instead argue as follows. Note that, $\Lambda \subset S^3\setminus D$ itself lives inside an open contact Darboux ball, which is contactomorphic to the standard contact vector space $\R^{ 3}$. Finally, we can use the fact that any Legendrian inside the standard contact vector space sits in a positive contractible loop by an explicit construction; see e.g.~\cite{Chantraine:Positive}.
\end{proof}

\begin{proof}[Proof of Theorem \ref{thm:c}]
Corollary \ref{cor:Lagcyl} takes a periodic orbit of the Liouville flow and produces an exact Lagrangian cylinder $L \subset (X,\lambda)$ with cylindrical ends in two different components of the contact boundary. Theorem \ref{thm:b} implies that $HW^*(L,L)\neq 0$. We argue by contradiction, and assume that the periodic orbit of the Liouville flow is contained inside a smooth ball. It follows that each Legendrian boundary component of $L$ also is contained inside a smooth ball inside the respective component of $\partial X=Y$.

Finally, since $Y=\partial X$ is a three-dimensional contact manifold that is tight (since it is fillable), it follows from Eliashberg's uniqueness result for tight 3-dimensional contact balls \cite{Contact3Man} that each of the two components of $\partial L$ is contained inside an open standard Darboux ball in $\partial Y$. Since the interior of any open Darboux ball is contactomorphic to $S^3 \setminus D$ where $D$ is a closed Darboux ball, we can invoke Corollary \ref{cor:vanishing} to conclude the vanishing of the wrapped Floer cohomology group $HW^*(L,L)=0$. This leads to the sought contradiction.
\end{proof}

\appendix
\section{ Action conventions, Floer strips, and neck-stretching}
\label{sec:appendix}

The goal here is to establish the ``neck stretching'' result Lemma \ref{lma:actionstretch}, which is similar to \cite[Lemma 6.2]{Dimitroglou:Energy}, and which can be used for the same purposes as \cite[Lemma 2.4]{CieliebakOancea}. The purpose of this technique is to exclude the existence of Floer-strips with certain asymptotics that cross a barrier in the form of a hypersurface of contact type. First, we need to introduce the conventions used in the setup of the Floer theories used here.

\subsection{Cylindrical almost complex structures}

Recall that a compact hypersurface $\Sigma^{2n-1} \subset (\hat{X}^{2n},\lambda)$ is of \textbf{contact type} if the Liouville vector field $\zeta$ is everywhere transverse to $\Sigma$. Near $\Sigma$ there are induced coordinates that identifies a subset of $(\hat{X}^{2n},\lambda)$ with
$$ \iota \colon ([-\epsilon,\epsilon]_\tau \times \Sigma^{2n-1},e^\tau\alpha_\Sigma) \hookrightarrow (\hat{X}^{2n},\lambda),\:\: \alpha_\Sigma \coloneqq \lambda|_{T\Sigma},$$
while preserving the primitives of the symplectic form, and where $\iota|_{\{0\} \times \Sigma}$ is the original embedding of $\Sigma$.

We say that a compatible almost complex structure is \textbf{cylindrical} near a hypersurface of contact type if is invariant under the Liouville flow (i.e.~invariant under translation of the $\tau$--coordinate), satisfies $J(\ker \alpha_\Sigma \cap T\Sigma)=\ker\alpha_\Sigma\cap T\Sigma$, and $J\partial_\tau=R$ where $R$ is the Reeb vector-field in $\Gamma(T\Sigma)$ induced by the contact-form $\alpha_\Sigma$. The most important feature of a cylindrical almost complex structure that we will use, is that it simultaneously is compatible also with any symplectic form $d(e^{\psi(\tau)} \alpha_\Sigma)$ where $\psi(\tau)$ is smooth and satisfies $\psi'(\tau) >0$.

We say that a Lagrangian $L$ is cylindrical near $\Sigma$ if it is of the form
$$[-\epsilon,\epsilon] \times \Lambda \subset [-\epsilon,\epsilon] \times \Sigma$$
in the coordinates above, which means that $\Lambda \subset (\Sigma,\alpha_\Sigma)$ is a Legendrian submanifold.

\subsection{Floer strips and their action}
\label{sec:floerstrips}
Our Floer complex are defined with differentials and continuation maps that count 
\textbf{finite energy rigid Floer strips}, i.e.~transversely cut-out solutions
\begin{equation}
 \label{eq:floerstrip}
\begin{cases}
u \colon (\R_s \times [0,1]_t,\R \times\{0\},\R \times\{1\}) \to (\hat{X},\hat{L}_0,\hat{L}_1),\\
\partial_s u+J_t(u)(\partial_t u-X_{ t}\circ u)=0.
\end{cases}
\end{equation}
Here $J_t$ is an auxiliary choice of time-dependent family of compatible almost complex structures that are fixed and cylindrical outside of a compact subset, and $X_{ t}$ is a $t$-dependent Hamiltonian vector field. The asymptotics $x_\pm(t)=\lim_{s \to \pm\infty} u(s,t)$ are Hamiltonian chords $x_\pm(t)=\phi^t_{H}(x_\pm(0))$, $t \in [0,1]$.

There is a similar equation in the case of periodic Hamiltonian orbits $\phi^t_H(x(0))$ where $t \in S^1$ thus is a coordinate on the circle, and the domain is an infinite cylinder instead of a strip. These so-called Floer cylinders are used when defining the symplectic cohomology complex.

The Floer differentials count solutions
$$ \langle d(x_+),x_-\rangle =\# \mathcal{M}(x_+,x_-)$$
in the moduli space $\mathcal{M}(x_+,x_-)$ of strips or cylinders of the above type for generic data of expected dimension one, with input given by the asymptotic $x_+$, while the output is the asytmptotic $x_-$.

There are also other operations that are necessary for defining Floer theories, such as continuation maps, and chain homotopy operators between continuation maps. They are defined similarly, but where the Hamiltonian vector-field $X_{s,t}$ in the Floer equation also is allowed to satisfy $\partial_s X_{s,t} \neq 0$ for $s \in [-A,A]$ in some bounded subset, i.e.~it is allowed to have a non-trivial dependence on $s$ there.

For the action of generators, we employ the convention from \cite{CieliebakOancea}. The action of an 1-periodic Hamiltonian orbit $x(t)=\phi^t_{H}(x(0)) \in SC^*(X,\lambda)$ in the Floer defined for the Hamiltonian $H$ is equal to
\begin{equation}
 \label{eq:actionclosed}
\mathcal{A}_{\lambda,H}(x) =\int_{x}\lambda-\int_0^1 H_t(x(t))dt.
\end{equation}
Similarly, the action of a time-1 Hamiltonian chord $x(t)=\phi^t_{H}(x(0)$ from $x(0) \in \hat{L}_0$ to $x(1) \in \hat{L}_1$ is equal to
\begin{equation}
 \label{eq:actionopen}
\mathcal{A}_{\lambda,H,\{f_i\}}(x) =\int_{x}\lambda-\int_0^1 H_t(x(t))dt+f_0(x(0))-f_1(x(1)),
\end{equation}
where $f_i \colon \hat{L}_i \to \R$ are auxiliary choices of \textbf{potentials}, i.e.~primitives of $\lambda|_{T\hat{L}_i}$. The potentials exist by the assumption that $\hat{L}_i$ are exact Lagrangian submanifolds. When the Floer energy
$$ E(u) \coloneqq \int_u d\lambda(\partial_s u, J\partial_s u) dsdt \ge 0 $$
is finite for a Floer strip (resp. cylinder) $u$, it follows that $u$ is asymptotic to a chord (resp. orbit) $x_+$ and $x_-$ at $s=+\infty$ and $s=-\infty$, respectively. Further, we have 
\begin{equation}\mathcal{A}_{\lambda,H}(x_+)-\mathcal{A}_{\lambda,H}(x_-)=E(u)\end{equation}
whenever $X_{s,t}=X_t$ is the $s$-independent Hamiltonian vector-field generated by $H_t$, e.g.~when $u$ is a Floer strip counted by the differential. In other words, the Floer differential \emph{decreases} the filtration induced by the action defined above.

The other important case is when $u$ is a Floer strip counted by the continuation map that increases the slope of the Hamiltonian $H_t$ with Hamiltonian vector field $X_t$, i.e.~$X_{s,t}=\beta(s)X_t$ for a bump-function $\beta(s) \ge 1$ with $\beta'\le 0$ of compact support, and $\beta(s)=1$ for $s \gg 0$. Here we let the Hamiltonian vector fields $X^\pm_t=\beta(\pm\infty)X_t$ be induced by Hamiltonians $H^\pm_t $, and thus we get
\begin{equation}
\mathcal{A}_{\lambda,H^+}(x_+)-\mathcal{A}_{\lambda,H^-}(x_-)\ge E(u)
\end{equation}

\subsection{Neck-stretching lemma based upon action}

Now assume that $\Sigma$ is a contact-type hypersurface in $(\hat{X},\lambda)$ that divides the latter Liouville domain into two components $\hat{X} \setminus \Sigma = \hat{X}_+ \sqcup \hat{X}_-$. Here we denote by $\hat{X}_\pm$ the component that contains the time-$\pm\epsilon$ Liouville flow of $\Sigma$.

The main result here is the following condition that prevents Floer-strips and Floer cylinders used in the definition of the differential, as well as certain continuation maps, from crossing the contact-type hypersurface $\Sigma$. 

\begin{lma}
\label{lma:actionstretch}
Assume that the Hamiltonian takes the form $H_t(\tau)=ae^{\tau}+b$ near $\Sigma$ for $a >0, b \ge 0$, and that $J$ is a compatible almost complex structure on $(\hat{X},\lambda)$ which is cylindrical near $\Sigma$ as well as outside of a compact subset.
\begin{itemize}
\item A Floer cylinder which is either a Floer differential or a continuation cylinder as above, whose input asymptotic is a periodic Reeb orbit $x_+$ contained in $\hat{X}_- \setminus \{\tau \ge -\epsilon\}$, cannot have an output asymptotic $x_-$ contained in $\hat{X}_+$ that is of \emph{positive} action;
\item The analogous statement also holds for a Floer strip asymptotic to Reeb chords from $\hat{L}_0$ to $\hat{L}_1$, under the additional assumption that $\hat{L}_i \subset \hat{X}$, $i=0,1$, are Lagrangian submanifolds that are cylindrical near $\Sigma$, where $\hat{L}_i$ admits a global primitive of $\lambda|_{\hat{T}L_i}$ that vanishes on $\Lambda_i = \hat{L}_i \cap \Sigma$, and where the action of the generators has been computed for these choices of potentials.
\end{itemize}
\end{lma}
\begin{rmk}
The remarkable feature of this result is that it does not require to pass to any limit, it just requires a particular choice of almost complex structure near a contact-type hypersurface $\Sigma \subset (\hat{X},\lambda)$. See \cite[Lemma 6.2]{Dimitroglou:Energy} for a similar result.
\end{rmk}
\begin{proof}
We deform the symplectic structure on $(\hat{X},\lambda)$ by stretching the neck along $\Sigma \subset \hat{X}$. This amounts to the following construction. Consider a smooth function $\psi_C \colon [-\epsilon,\epsilon] \to \R$ that satisfies
\begin{itemize}
\item $\psi_C'(\tau) > 0$ for all $\tau \in [-\epsilon,\epsilon]$;
\item $\psi_C(\tau)=\tau$ near $\tau=-\epsilon$; and
\item $\psi_C(\tau)=C+\tau$ on $\tau \in [0,\epsilon]$ for some $C \ge 1$.
\end{itemize}
We obtain a family $(\hat{X},\lambda_C)$ of Liouville manifolds where $\lambda_C$ is determined in the following manner
\begin{itemize}
\item $\lambda_C=\lambda$ in $\hat{X}_- \setminus \{\tau \ge -\epsilon\}$;
\item $\lambda_C=e^C\lambda$ in $\hat{X}_+ \setminus \{\tau \le \epsilon\}$;
\item $\iota^*\lambda_C=e^{\psi_C(\tau)}\alpha$ in the tubular neighbourhood $[-\epsilon,\epsilon]_\tau \times \Sigma$ of $\Sigma$.
\end{itemize}

The Hamiltonian vector-field $X_t$ on $(\hat{X},d\lambda)$ that is generated by $H_t$ is still a Hamiltonian vector field when considered on $(\hat{X},d\lambda_C)$. The corresponding Hamiltonian coincides with $H_t$ on $\hat{X}_- \setminus \{\tau \ge -\epsilon\}$. Further, since the Hamiltonian vector field is given by $aR$ on $[-\epsilon,\epsilon] \times \Sigma$, it is generated by the Hamiltonian $ ae^{-\epsilon}+b+ \int_{-\epsilon}^\tau ae^{\psi_C(t)}dt$ there.

We conclude that the Hamiltonian $H^C_t$ on $(\hat{X},d\lambda_C)$ that generates $X_t$ satisfies
$$H^C_t=e^C H_t- e^{C}(ae^\epsilon+b)+(ae^{-\epsilon}+b)+\int_{-\epsilon}^\epsilon ae^{\psi_C(t)}dt$$
on $\hat{X}_+$. If we choose $\epsilon>0$ sufficiently small and $C \gg 0$ sufficiently large, the action of a generator $x_-$ in $\hat{X}_+$ thus satisfies
$$\mathcal{A}_{\lambda_C,H^C_t}(x_-) \ge e^C \mathcal{A}_{\lambda,H_t}(x_-)$$
for $C \gg 0$, while the action of a generator $x_+$ in $\hat{X}_- \setminus \{\tau \ge -\epsilon\}$ satisfies $\mathcal{A}_{\lambda_C,H^C_t}(x_+)=\mathcal{A}_{\lambda,H_t}(x_+)$.

Choosing $C \gg 0$ sufficiently large implies that there can be no Floer strip with input $x_+$ and output $x_-$ for either complexes $SC^*(X;H_t)$ or $CW^*(L_0,L_1;H_t)$; in the latter case of the wrapped Floer homology complex, it is necessary to use the choices of potentials of $\lambda|_{T\hat{L}_i}$ and $\lambda_C|_{T\hat{L}_i}$ that vanish on $\Lambda_i=\hat{L}_i \cap \Sigma$ and which hence also vanish in the cylindrical neighbourhood.
\end{proof}

\bibliographystyle{alphanum}
\bibliography{references}

\def\cprime{$'$}
\begin{thebibliography}{CDRGG}

\bibitem[AS]{Abouzaid:OpenString}
M.~Abouzaid and P.~Seidel.
\newblock An open string analogue of {V}iterbo functoriality.
\newblock {\em Geom. Topol.}, 14(2):627--718, 2010.

\bibitem[Aud]{Audin}
M.~Audin.
\newblock Fibr\'es normaux d'immersions en dimension double, points doubles
  d'immersions lagragiennes et plongements totalement r\'eels.
\newblock {\em Comment. Math. Helv.}, 63(4):593--623, 1988.

\bibitem[CCDR]{Chantraine:Positive}
B.~Chantraine, V.~Colin, and D.~Dimitroglou~Rizell.
\newblock Positive {L}egendrian isotopies and {F}loer theory.
\newblock {\em Ann. Inst. Fourier (Grenoble)}, 69(4):1679--1737, 2019.

\bibitem[CDRGG]{Generation}
B.~Chantraine, G.~Dimitroglou~Rizell, P.~Ghiggini, and R.~Golovko.
\newblock Geometric generation of the wrapped {F}ukaya category of {W}einstein
  manifolds and sectors.
\newblock {\em Ann. Sci. \'Ec. Norm. Sup\'er. (4)}, 57(1):1--85, 2024.

\bibitem[CE]{Cieliebak:SteinWeinstein}
K.~Cieliebak and Y.~Eliashberg.
\newblock {\em {From Stein to Weinstein and back. Symplectic geometry of affine
  complex manifolds.}}
\newblock {Colloquium Publications. American Mathematical Society 59.
  Providence, RI: American Mathematical Society (AMS). 354~p.}, 2012.

\bibitem[Cie]{Cieliebak:Handle}
K.~Cieliebak.
\newblock Handle attaching in symplectic homology and the chord conjecture.
\newblock {\em Journal of the European Mathematical Society}, 004(2):115--142,
  2002.

\bibitem[CLMM]{Cieliebak:Anasov}
K.~Cieliebak, O.~Lazarev, T.~Massoni, and A.~Moreno.
\newblock Floer theory of {A}nosov flows in dimension three.
\newblock Preprint, \href{https://arxiv.org/abs/2211.07453 }{2211.07453
  [math.SG]}, 2022.

\bibitem[CO]{CieliebakOancea}
K.~Cieliebak and A.~Oancea.
\newblock Symplectic homology and the {E}ilenberg–-{S}teenrod axioms.
\newblock {\em Algebr. Geom. Topol.}, 18(4):1953–--2130, 2018.

\bibitem[DR]{Dimitroglou:Whitney}
G.~Dimitroglou~Rizell.
\newblock The classification of {L}agrangians nearby the {W}hitney immersion.
\newblock {\em Geom. Topol.}, 23(7):3367--3458, 2019.

\bibitem[DRET]{DimitroglouRizell2022}
G.~Dimitroglou~Rizell, T.~Ekholm, and D.~Tonkonog.
\newblock Refined disk potentials for immersed {L}agrangian surfaces.
\newblock {\em J. Differential Geom.}, 121(3):459--539, 2022.

\bibitem[DRGI]{Dimitroglou:Isotopy}
G.~Dimitroglou~Rizell, E.~Goodman, and A.~Ivrii.
\newblock Lagrangian isotopy of tori in {$S^2\times S^2$} and
  {$\mathbb{C}P^2$}.
\newblock {\em Geom. Funct. Anal.}, 26(5):1297--1358, 2016.

\bibitem[DRS]{Dimitroglou:Energy}
G.~Dimitroglou~Rizell and M.~G. Sullivan.
\newblock An energy-capacity inequality for {L}egendrian submanifolds.
\newblock {\em J. Topol. Anal.}, 12(3):547--623, 2020.

\bibitem[EKP]{GeometryOrderability}
Y.~Eliashberg, Sang~Seon Kim, and L.~Polterovich.
\newblock Geometry of contact transformations and domains: orderability versus
  squeezing.
\newblock {\em Geom. Topol.}, 10:1635--1747, 2006.

\bibitem[Eli]{Contact3Man}
Y.~Eliashberg.
\newblock Contact {$3$}-manifolds twenty years since {J}. {M}artinet's work.
\newblock {\em Ann. Inst. Fourier (Grenoble)}, 42(1-2):165--192, 1992.

\bibitem[EO]{EkholmOancea}
T.~Ekholm and A.~Oancea.
\newblock Symplectic and contact differential graded algebras.
\newblock {\em Geometry and Topology}, 21(4):2161--2230, 2017.

\bibitem[ET]{Confoliations}
Y.~Eliashberg and W.~P. Thurston.
\newblock Confoliations.
\newblock In {\em Collected works of {W}illiam {P}. {T}hurston with commentary.
  {V}ol. {I}. {F}oliations, surfaces and differential geometry}, pages
  281--351. Amer. Math. Soc., Providence, RI, [2022] \copyright 2022.
\newblock Reprint of [1483314].

\bibitem[Flo]{Floer:Morse}
A.~Floer.
\newblock Morse theory for {L}agrangian intersections.
\newblock {\em J. Differential Geom.}, 28(3):513--547, 1988.

\bibitem[FOOO]{FOOO:I}
K.~Fukaya, Y.-G. Oh, H.~Ohta, and K.~Ono.
\newblock {\em Lagrangian intersection {F}loer theory: anomaly and obstruction.
  {P}art {I}}, volume~46 of {\em AMS/IP Studies in Advanced Mathematics}.
\newblock American Mathematical Society, Providence, RI, 2009.

\bibitem[Gei]{Geiges:Intro}
H.~Geiges.
\newblock {\em An introduction to contact topology}, volume 109 of {\em
  Cambridge Studies in Advanced Mathematics}.
\newblock Cambridge University Press, Cambridge, 2008.

\bibitem[Gir]{Giroux}
E.~Giroux.
\newblock Convexit\'{e} en topologie de contact.
\newblock {\em Comment. Math. Helv.}, 66(4):637--677, 1991.

\bibitem[GPS]{GPS}
S.~Ganatra, J.~Pardon, and V.~Shende.
\newblock Sectorial descent for wrapped {F}ukaya categories.
\newblock {\em J. Amer. Math. Soc.}, 37(2):499--635, 2024.

\bibitem[Gro]{Gromov:Pseudo}
M.~Gromov.
\newblock Pseudoholomorphic curves in symplectic manifolds.
\newblock {\em Invent. Math.}, 82(2):307--347, 1985.

\bibitem[HCK]{HedickeCantKilgore}
J.~Hedicke, D.~Cant, and E.~Kilgore.
\newblock Extensible positive loops and vanishing of symplectic cohomology.
\newblock Preprint, \href{https://arxiv.org/pdf/2311.18267}{2311.18267
  [math.SG]}, 2023.

\bibitem[Hoz1]{Hozoori:Skeleton}
S.~Hozoori.
\newblock Regularity and persistence in non-{W}einstein {L}iouville geometry
  via hyperbolic dynamics.
\newblock Preprint, \href{https://arxiv.org/abs/2409.15592}{2409.15592
  [math.SG]}, 2024.

\bibitem[Hoz2]{Hozoori}
S.~Hozoori.
\newblock Symplectic geometry of {A}nosov flows in dimension 3 and bi-contact
  topology.
\newblock {\em Adv. Math.}, 450:Paper No. 109764, 41, 2024.

\bibitem[HS]{HedickeShelukhin}
J.~Hedicke and E.~Shelukhin.
\newblock Non-orderability and the contact {H}ofer norm.
\newblock Preprint, \href{https://arxiv.org/abs/2411.19887}{2411.19887
  [math.SG]}, 2024.

\bibitem[Hus]{Husin}
A.~Husin.
\newblock Maslov class of exact {L}agrangians and cylindrical handles.
\newblock Preprint, \href{https://arxiv.org/abs/2502.14750}{2502.14750
  [math.SG]}, 2025.

\bibitem[Jac]{Jaco}
W.~Jaco.
\newblock On certain subgroups of the fundamental group of a closed surface.
\newblock {\em Proc. Cambridge Philos. Soc.}, 67:17--18, 1970.

\bibitem[LS]{LalondeSikorav:SousVarietees}
F.~Lalonde and J.-C. Sikorav.
\newblock Sous-vari\'et\'es lagrangiennes et lagrangiennes exactes des fibr\'es
  cotangents.
\newblock {\em Comment. Math. Helv.}, 66(1):18--33, 1991.

\bibitem[Mas1]{Massoni:Taut}
T.~Massoni.
\newblock Taut foliations and contact pairs in dimension three.
\newblock Preprint, \href{https://arxiv.org/abs/2405.15635}{2405.15635
  [math.SG]}, 2022.

\bibitem[Mas2]{Massoni:Anosov}
Thomas Massoni.
\newblock Anosov flows and {L}iouville pairs in dimension three.
\newblock {\em Algebr. Geom. Topol.}, 25(3):1793--1838, 2025.

\bibitem[McD]{McDuff}
D.~McDuff.
\newblock Symplectic manifolds with contact type boundaries.
\newblock {\em Invent. Math.}, 103(3):651--671, 1991.

\bibitem[Mit]{Mitsumatsu}
Y.~Mitsumatsu.
\newblock Anosov flows and non-{S}tein symplectic manifolds.
\newblock {\em Ann. Inst. Fourier (Grenoble)}, 45(5):1407--1421, 1995.

\bibitem[Rit]{Ritter:TQFT}
A.~F. Ritter.
\newblock Topological quantum field theory structure on symplectic cohomology.
\newblock {\em J. Topol.}, 6(2):391--489, 2013.

\bibitem[Sco]{Scott}
P.~Scott.
\newblock Subgroups of surface groups are almost geometric.
\newblock {\em Journal of the London Mathematical Society}, s2-17(3):555--565,
  1978.

\bibitem[She]{LagKlein}
V.~V. Shevchishin.
\newblock Lagrangian embeddings of the {K}lein bottle and the combinatorial
  properties of mapping class groups.
\newblock {\em Izv. Ross. Akad. Nauk Ser. Mat.}, 73(4):153--224, 2009.

\bibitem[Zun]{Zung}
J.~Zung.
\newblock Reeb flows transverse to foliations.
\newblock {\em Geom. Topol.}, 28(8):3661--3695, 2024.

\end{thebibliography}

\end{document}